\documentclass[10pt,a4paper]{article}
\usepackage[utf8]{inputenc}
\usepackage[T1]{fontenc}
\usepackage{amsmath}
\usepackage{amsthm}
\usepackage{amsfonts}
\usepackage{amssymb}
\usepackage{bbm}
\usepackage{a4,a4wide}
\usepackage{graphicx}
\usepackage{pstricks}
\usepackage{color}
\usepackage{todonotes}
\definecolor{darkblue}{rgb}{0,0,0.6}
\definecolor{darkred}{rgb}{0.6,0,0}
\usepackage[colorlinks=true,urlcolor=darkblue,citecolor=darkblue,linkcolor=darkred]{hyperref}
\usepackage[bibtex-style,initials,nobysame]{amsrefs}
\usepackage[only,llbracket,rrbracket]{stmaryrd}

\newif\ifscreen
\ifscreen
\usepackage{cmbright}

\usepackage[italic,defaultmathsizes,basic,selfGreek]{mathastext}
\DeclareSymbolFont{letterg}{OML}{cmbrm}{m}{it}
\DeclareMathSymbol{g}{\mathalpha}{letterg}{`g}
\fi

\newtheorem{theo}{\textbf{Theorem}}[section]
\newtheorem{lem}[theo]{\textbf{Lemma}}
\newtheorem{prop}[theo]{\textbf{Proposition}}

\newtheorem{exa}[theo]{\textbf{Example}}
\newtheorem{rem}[theo]{\textbf{Remark}}
\newtheorem{ass}[theo]{\textbf{Assumption}}

\title{Local stability of perfect alignment for a spatially homogeneous kinetic model}
\author{Pierre Degond\footnote{Department of Mathematics, Imperial College, South Kensington Campus, London SW7 2AZ, UK. E-mail: p.degond@imperial.ac.uk},\: Amic Frouvelle\footnote{CEREMADE, UMR 7534, Université Paris--Dauphine, Place du Maréchal de Lattre de Tassigny, 75775 Paris Cedex 16, France. E-mail: frouvelle@ceremade.dauphine.fr}\: and Gaël Raoul\footnote{Centre d'Écologie Fonctionnelle et Évolutive, UMR
5175, CNRS, 1919 Route de Mende, 34293 Montpellier, France.
E-mail: raoul@cefe.cnrs.fr.}}
\date{}

\begin{document}

\maketitle

\begin{abstract}
We prove the nonlinear local stability of Dirac masses for a kinetic model of alignment of particles on the unit sphere, each point of the unit sphere representing a direction. A population concentrated in a Dirac mass then corresponds to the global alignment of all individuals. The main difficulty of this model is the lack of conserved quantities and the absence of an energy that would decrease for any initial condition. We overcome this difficulty thanks to a functional which is decreasing in time in a neighborhood of any Dirac mass (in the sense of the Wasserstein distance).
The results are then extended to the case where the unit sphere is replaced by a general Riemannian manifold.
\end{abstract}

\section{Introduction}\label{section:intro}
Alignment mechanisms are present in many biological and physical systems such as self-propelled particles or rod-like polymers. One of the relevant questions in the study of such systems is to know whether the interactions between individuals allow the whole system to become globally aligned. One can then talk about flocking, long range polarization, magnetization, etc. depending on the context. In this work we investigate a simple kinetic model with an alignment mechanism for which this flocking behaviour can be rigorously exhibited.
 
We consider an infinite number of individuals/particles structured by their orientation in the unit sphere of~$\mathbb{R}^n$. The model we  consider here is homogeneous in the sense that the density of particles is not spatially structured. We assume that all particles interact with the same probability, and when they do, their new orientation is the average orientation of the particles prior to the interaction. The rate at which a particle is chosen to interact with another one is constant (equal to 1) and independent of the particle. Notice that when the interacting particles have opposite directions, their ``average direction'' is not well-defined. Their new orientations are then chosen at random on the corresponding equator (the set of directions orthogonal to the initial opposite directions).
 
This model is related to the Boltzmann-like model of Bertin, Droz and Grégoire~\cites{bertin2006boltzmann,bertin2009hydrodynamic} in which self-propelled particles interact and tend to align. When two particles interact, they align following the process described above, and their orientation is then subject to angular noise. The model we consider here is then a kinetic model corresponding to the spatially homogeneous version of this model, when no noise is present in the system. The links between the kinetic model and the underlying individual based model have been studied in~\cites{carlen2013kinetic,carlen2013hierarchy}. In this study, we will focus on the asymptotic behavior of the solutions to the spatially homogeneous kinetic model. It is indeed well known that the study of spatially inhomogeneous Boltzmann-type models is a much harder problem, and the analysis of such equations (through hydrodynamic limits, for instance) requires a good understanding of the corresponding homogeneous model. From the modeling point of view, the study of such models is also a step towards the understanding of more realistic models based on the fact that birds in a flock interact with a limited number of congeners~\cite{ballerini2008interaction}. The model of~\cites{bertin2006boltzmann,bertin2009hydrodynamic} is intended to mimic the Vicsek alignment model~\cite{vicsek1995novel}. While the model of~\cites{bertin2006boltzmann,bertin2009hydrodynamic} relies on a binary interaction point of view, that of Vicsek describes multiple simultaneous interactions. The corresponding kinetic model is of mean-field type and takes the form of a Fokker--Planck equation~\cite{bolley2012meanfield}. Equilibria of this Fokker--Planck model have been studied in~\cites{frouvelle2012dynamics,degond2014phase} in the space-homogeneous case, and the non-spatially homogeneous case is addressed in~\cites{degond2008continuum, degond2010macroscopic, degond2013macroscopic, degond2014phase}.
 
We then consider two generalizations of the model. The first generalization that we will consider is a more general ``position space'' to describe the orientation of the particles. From the modeling point of view, it is indeed interesting to consider more complex position spaces. The main example we have in mind concerns elongated particles which have an orientation (that can be seen as a point on the unit sphere), but are symmetrical with respect to their center of mass. A particle with a given orientation is the indistinguishable from a particle with the opposite direction, so that those particles are indeed characterized by a ``position'' that is a point on the sphere quotiented by~$\{\pm\mathrm{Id}\}$. The alignment dynamics is also affected: the particles try to align to neighboring particles,  but possibly with opposite directions. A good modeling choice then seems to describe the dynamics of the particles as points on the Riemannian manifold formed by the sphere quotiented by~$\{\pm\mathrm{Id}\}$. This assumption has for instance been chosen to model the alignment dynamics of fish in~\cite{carrillo2009double}, leading to possible contra-rotating mills (that have indeed been observed on natural populations). A second example is flocking birds, that do not only copy their neighbors' orientation, but also their posture (in particular the angular orientation of their wings). It can then be interesting to characterize the position of the birds through a more complicated manifold. We will thus extend our result on the local stability of Dirac masses to alignment models on general Riemannian manifolds. In that case, the interaction between two particles consists in updating their position at the midpoint of the minimal geodesic joining their ``positions'' prior to the interaction. The second generalization that we make is to relax our assumptions on the outcome of the interactions: instead of assuming that the two interacting particles align perfectly, we merely assume that the interaction brings them closer. 
 
When the space of positions is the euclidean space~$\mathbb{R}^n$,~$n=1,\,2,\,3$, the model we study here is related to several other models from physics of sociology. We can for instance mention the inelastic Boltzmann equation~\cites{bobylev2000properties,mischler2009stability,carrillo2009overpopulated}, used for instance to model granular gases, or the simpler Kac inelastic model~\cites{pulvirenti2004asymptotic,gabetta2012complete,bassetti2011central}. This type of model also appears in socio-physical models~\cites{boudin2009kinetic,galam1982sociophysics}, or wealth redistribution models in agent-based markets \cites{bisi2009kinetic,che2011kinetic}. Finally, such models appear in biology, to model the effect of recombination in sexual population (in the so-called infinitesimal model)~\cites{bulmer1980mathematical,turelli1994genetic}, or the exchange of proteins between cells~\cites{pasquier2012different,hinow2009analysis}. In all those works, the fact that the ``position space'' is~$\mathbb{R}^n$ allows the development of powerful methods. In particular, the center of mass is conserved, and it is possible to define an energy that is decreasing for any initial condition. Moreover, thanks to the geometry of the position space, it is possible to use Fourier-based methods as well as the convexity properties obtained  thanks to optimal transport methods (see~\cite{carrillo2007contractive} for a synthetic description of those methods). It is then possible to obtain precise descriptions of the dynamics of solutions. Unfortunately, it turns out that the properties and methods listed above fail on more general spaces of positions, such as spheres, as we consider in this article. 
One could hope that the Fourier transform methods introduced for e. g. the Kac model (see e.g.~\cite{carrillo2007contractive}) could be reproduced in the case of spheres using spherical harmonics. This approach is currently under investigation in the case of the unit circle, see~\cite{carlen2014model}, but it appears that severe difficulties arise compared to the euclidian situation.
 
The main result of this article is that the solutions of the kinetic model we consider converge exponentially fast to a Dirac mass, provided the initial condition is close to a Dirac mass (in the sense of the 2-Wasserstein distance). We first provide a proof of the main result in the case of the so-called midpoint model on a sphere (which is the simplest model for alignment in our setting), in order to keep the technicality of the proof to a minimum. Once our main result is proven in this simple case, we show that this initial proof can actually be extended to general Riemannian manifolds, thanks to some geometrical arguments that we detail.  Finally, we show that more general interactions can be considered, thanks to additional geometrical arguments that we provide.
 
The main difficulty of our study is the absence of conserved quantities which would help to locate the asymptotic final Dirac mass, and that could play the role of the center of mass in the inelastic Boltzmann equation. Similarly, we do not believe there exists an energy functional that would decrease in time for any initial condition.  This lack of conserved quantities and global Lyapunov functional is common in models describing biological systems. It is however sometimes possible for such models to prove the convergence of a solution to a stationary profile in large time, even if it is not possible to determine  which specific asymptotic profile will be selected~\cites{frouvelle2012dynamics,giacomin2012global}. A second difficulty of the model comes from the geometry of the position space, which seems to prevent the use of arguments based on convexity. 
 
The main argument we use in our analysis is an energy adapted to the problem, which is a Lyapunov functional provided the initial condition is sufficiently close to a Dirac mass. To prove  that this energy is actually decreasing, we introduce some local and global estimates regarding the microscopic variation of this energy. Those estimates require a good understanding of the geometry of the problem, involving a positive injectivity radius and bounds on the curvature of the manifold.
 
The plan of the paper is the following. In Section~\ref{section-model}, we define the kinetic model and discuss its well-posedness. We detail the case of the midpoint model on the euclidean space~$\mathbb{R}^n$, for which the convergence towards a Dirac mass can be easily proven, and we detail the difficulties appearing for more general position sets, describing in particular the case where the position space is the circle~$\mathbb{S}^1$. In Section~\ref{section:midpoint-sphere}, we prove our main result in the case of the midpoint model on the unit sphere of~$\mathbb{R}^n$. In Section~\ref{section:manifold}, we extend this result to the case of a general Riemannian manifold. Finally, in Section~\ref{section:non-midpoint}, we define the notion of contracting models, and we extend our main result to this wider class of models.

\section{The model and preliminary results}\label{section-model}
\subsection{The model}

We are interested in the evolution of a set of particles, represented at each time~$t\geqslant 0$ by a probability measure~$\rho(t,\cdot)$ over a set of positions~$x\in\mathcal M$. See Section~\ref{section:intro} for typical examples of position sets. The position set is here a complete connected Riemannian manifold~$\mathcal M$, endowed with the geodesic distance~$d$. The particles collide, and the effect of the collisions is described by a collision kernel~$K$: given two pre-collisional positions~$x_*,x_*'$ in~$\mathcal M$, the law of the post-collisional position of the particles is the probability law~$K(\cdot,x_*,x_*')$ on~$\mathcal M$. We assume that the rate at which a particle is colliding with another is constant equal to~$1$, and that particles collide independently of their positions.



Since additionally we consider an infinite number of particles, the model we are interested in is the following, where the unknown~$\rho(t,x)$ is the probability density of finding particles at time~$t$ and position~$x$:
\begin{equation}
\label{eq-kinetic-model}
\partial_t \rho(t,x)=\int_{\mathcal M\times\mathcal M} K(x,x_*,x_*')\rho(t,x_*)\rho(t,x_*')\,dx_*\,dx_*'-\rho(t,x), 
\end{equation}
where we use the abusive notation~$\rho(t,x_*)\,dx_*$ instead of~$d\left(\rho(t,\cdot)\right)(x_*)$.
When no confusion is possible, we will write~$\rho$ or~$\rho_t$ for~$\rho(t,\cdot)$ and~$\rho(x)$ for~$\rho(t,x)$.
We can consider that the particle are indistinguishable and therefore~$K(\cdot,x_*,x_*')=K(\cdot,x_*',x_*)$, since symmetrizing~$K$ does not change the outcome of expression~\eqref{eq-kinetic-model}.
We therefore introduce the symmetric operator~$A$ as follows:
\begin{equation}\label{def-A}
A(\rho,\widetilde{\rho})(x)=\int_{\mathcal M\times\mathcal M} K(x,x_*,x_*')\,d\rho(x_*)\,d\widetilde{\rho}(x_*').
\end{equation}
Then~\eqref{eq-kinetic-model} is equivalent to~$\partial_t \rho=A(\rho,\rho)-\rho$ or, given in an integral form, for an initial condition~$\rho_0$:
\begin{equation}\label{eq-kinetic-model-integral}
\rho_t=\rho_0+\int_0^t[A(\rho_s,\rho_s)-\rho_s] ds.
\end{equation}
The framework in which we will enounce all the results is~$\mathcal P_2(\mathcal M)$, the set of probability measures~$\rho$ such that~$\int_{\mathcal M}d(x,\bar x)^2\,d\rho(x)dx<+\infty$ for some (or equivalently, for all)~$\bar x\in\mathcal M$, equipped with the Wasserstein distance~$W_2$ given by
\begin{equation}
W_2(\rho,\widetilde{\rho})^2=\inf_{\pi\in\Pi(\rho,\widetilde{\rho})}\int_{\mathcal M\times\mathcal M}d(x,y)^2\,d\pi(x,y)\label{defW2},
\end{equation}
where~$\Pi(\rho,\widetilde{\rho})$ is the set of probability measures on~$\mathcal M\times\mathcal M$ with marginals~$\rho$ and~$\widetilde{\rho}$. A first useful remark is that for any~$y\in\mathcal M$ and~$\rho\in\mathcal P(\mathcal M)$, we have
\begin{equation}
W_2(\rho,\delta_{y})^2=\int_{\mathcal M} d(x,y)^2\,d\rho(x),\label{eq-W2-dirac}
\end{equation}
since any probability measure on~$\mathcal M\times\mathcal M$ with marginals~$\rho$ and~$\delta_y$ is actually~$\rho\otimes\delta_y$.

\subsection{Existence and uniqueness of solutions}

The first proposition concerns the well-posedness of~\eqref{eq-kinetic-model} on~$\mathcal P_2(\mathcal M)$, under Lipschitz conditions for~$K$.
\begin{prop}\label{prop-existence-uniqueness}
Assume that for any~$x_*\in\mathcal M$, the map~$x_*'\mapsto K(\cdot,x_*,x_*')$ is a~$k$-Lipschitz map from~$\mathcal M$ to~$\mathcal P_2(\mathcal M)$, for some~$k$ independent of~$x_*$.
  Then, for any initial data~$\rho_0\in P_2(\mathcal M)$, there exists a unique global solution~$\rho\in C(\mathbb{R}_+,\mathcal P_2(\mathcal M))$ to~\eqref{eq-kinetic-model-integral}.
\end{prop}
\begin{proof}[Proof of Proposition~\ref{prop-existence-uniqueness}]
We recall the Kantorovich duality formula~\cite{villani2009optimal}:
\begin{equation}
W_2(\rho,\widetilde{\rho})^2=\sup_{(\varphi,\widetilde{\varphi})\in\Phi}\left(\int_{\mathcal M}\varphi(x)\,d\rho(x)+\int_{\mathcal M}\widetilde{\varphi}(y)\,d\widetilde{\rho}(y)\right),\label{dualityW2}
\end{equation}
where~$\Phi$ is the set of pairs of bounded continuous functions~$\varphi,\widetilde{\varphi}$ from~$\mathcal M$ to~$\mathbb{R}$ such that for all~$x,y$ in~$\mathcal M$ we have~$\varphi(x)+\widetilde{\varphi}(y)\leqslant d(x,y)^2$. 
For~$\mu$ in~$\mathcal P_2(\mathcal M)$, we first show that~$\rho \mapsto A(\mu,\rho)$ is~$k$-Lipschitz with respect to the~$W_2$ distance. Indeed, for~$(\varphi,\widetilde{\varphi})$ in~$\Phi$, for~$\rho$ and~$\widetilde{\rho}$ in~$\mathcal P_2(\mathcal M)$, and for~$\pi$ in~$\Pi(\rho,\widetilde{\rho})$, we estimate
\begin{align*}
 I(\varphi,\widetilde{\varphi})&\overset{def}{=}\int_{\mathcal M}\varphi(x)\,d[A(\mu,\rho)](x)+\int_{\mathcal M}\widetilde{\varphi}(y)\,d[A(\mu,\widetilde{\rho})](y)\\
&=\int_{\mathcal M^3}\varphi(x)K(x,x_*,x_*')\,dx\,d\mu(x_*)\,d\rho(x_*')+\int_{\mathcal M^3}\widetilde{\varphi}(y)K(y,x_*,y_*')\,dy\,d\mu(x_*)\,d\widetilde{\rho}(y_*')\\
 &=\int_{\mathcal M^3}\bigg(\int_{\mathcal M}\varphi(x)K(x,x_*,x_*')\,dx+\int_{\mathcal M}\widetilde{\varphi}(y)K(y,x_*,y_*')\,dy\bigg)\,d\mu(x_*)\,d\pi(x_*',y_*')\\
&\leqslant\int_{\mathcal M^3}W_2(K(\cdot,x_*,x_*'),K(\cdot,x_*,y_*'))^2\,d\mu(x_*)\,d\pi(x_*',y_*')\\
&\leqslant k^2\int_{\mathcal M^3}d(x_*',y_*')^2\,d\mu(x_*)\,d\pi(x_*',y_*'),
\end{align*}
and then, since~$\mu$ is a probability measure, if we take the infimum of the above quantity in~$\pi$, we get:
\begin{align*}
W_2(A(\mu,\rho),A(\mu,\widetilde{\rho}))^2=\sup_{(\varphi,\widetilde{\varphi})\in\Phi}I(\varphi,\widetilde{\varphi})\leqslant k^2\,W_2(\rho,\widetilde{\rho}).
\end{align*}
Therefore, using symmetry, we immediately get that~$\rho\mapsto A(\rho,\rho)$ is~$2k$-Lipschitz with respect to~$W_2$, and takes values in~$\mathcal P_2(\mathcal M)$, since for instance~$A(\delta_{\bar x},\delta_{\bar x})=K(\cdot,\bar x,\bar x)$ is in~$\mathcal P_2(\mathcal M)$. This implies that for~$\rho_0$ in~$\mathcal P_2(\mathcal M)$, the map~$F$ given by
\[F(\rho)_t=\rho_0e^{-t}+\int_0^te^{s-t}A(\rho_s,\rho_s)\,ds\]
sends~$C([0,T],\mathcal P_2(\mathcal M))$ into itself, and is a contraction for~$T$ sufficiently small (but independent of~$\rho_0$). Since~$\mathcal P_2(\mathcal M)$ is complete, then~$C([0,T],\mathcal P_2(\mathcal M))$ is also complete and the Banach fixed point theorem implies the existence of a unique fixed point of~$F$, which corresponds to a unique solution of~\eqref{eq-kinetic-model-integral} thanks to the Duhamel formula. It is then possible to extend~$\rho$ for any~$t>T$ (one simply considers recursively the above construction with~$\rho_{0,n}:=\rho_{nT}$, since~$T$ does not depend on~$\rho_0$).
\end{proof}

Not all interesting models satisfy the assumption of the above Proposition~\ref{prop-existence-uniqueness}, namely the fact that~$x_*'\mapsto K(\cdot,x_*,x_*')$ is Lipschitz. For instance, the midpoint model, where two colliding particles end up in the middle of a minimal geodesic joining the two pre-collisional velocities (see Subsection~\ref{subsection:Rn} and Section~\ref{section:midpoint-sphere}), satisfies the assumption of Proposition~\ref{prop-existence-uniqueness} if~$\mathcal M=\mathbb{R}^n$ (see Subsection~\ref{subsection:Rn}), but does not satisfy it if~$\mathcal M=\mathbb{S}^n$, which corresponds to the model studied in Section~\ref{section:midpoint-sphere}. Indeed, we even show in the example below that there exists no well-posed measure solutions in the latter case (we restrict the example to the case~$n=1$).

\begin{exa}\label{example1}
We consider here~$\mathcal M=\mathbb{S}^1$. 
Each position~$x\in \mathcal M$ can then be represented by an angle~$\theta\in [0,2\pi)$, considering that~$\mathcal M\subset\mathbb{C}$ and writing~$x=e^{i\theta}$. For~$\theta_*,\theta_*'\in  [0,2\pi)$,~$\theta_*\neq\theta_*'\pm\pi$, the middle of the minimal geodesic joining~$e^{i\theta_*}$ to~$e^{i\theta_*'}$ on the circle is given by~$e^{i\theta}$ where~$\theta=\frac{\theta_*+\theta_*'}2$ if~$|\theta_*-\theta_*'|<\pi$, and~$\theta=\frac{\theta_*+\theta_*'}2\pm\pi$ if~$|\theta_*-\theta_*'|>\pi$. We can then define the following midpoint collision operator~$K$:
\[K(\cdot,\theta_*,\theta_*')=\delta_{\theta},\text{ where } \theta=
  \begin{cases}\frac{\theta_*+\theta_*'}2&\text{if}\quad|\theta_*-\theta_*'|<\pi\\
  \frac{\theta_*+\theta_*'}2+\pi&\text{if}\quad|\theta_*-\theta_*'|>\pi\quad\text{and}\quad\frac{\theta_*+\theta_*'}2<\pi\\
  \frac{\theta_*+\theta_*'}2-\pi&\text{if}\quad|\theta_*-\theta_*'|>\pi\quad\text{and}\quad\frac{\theta_*+\theta_*'}2\geqslant\pi,
  \end{cases}\]
if~$|\theta_*-\theta_*'|\neq \pi$, while if~$|\theta_*-\theta_*'|=\pi$,
\[K(\cdot,\theta_*,\theta_*')=\frac12\delta_{\theta}+\frac12\delta_{\theta'},\text{ where } \theta=\frac{\theta_*+\theta_*'}2 \text{ and } \theta'=
  \begin{cases}\theta+\pi&\text{if}\quad\theta<\pi\\\theta-\pi&\text{if}\quad\theta\geqslant\pi.\end{cases}\]
We notice that if the support of the initial data~$\rho_0$ is supported in an open semicircle, then the support of the solution remains in this semicircle at all times. For instance, if~$\mathrm{supp}\,\rho_0\subset [0,\pi)$ (resp.~$\mathrm{supp}\,\rho_0\subset (\pi,2\pi)\cup \{0\}$), then for all~$t\geqslant 0$,~$\mathrm{supp}\,\rho_t\subset [0,\pi)$ (resp.~$\mathrm{supp}\,\rho_t\subset (\pi,2\pi)\cup \{0\}$). Moreover, in this case, the dynamics of~$t\mapsto \rho_t$ is given by midpoint model in~$\mathbb{R}$ (see Subsection~\ref{subsection:Rn}). 
%
%
%
%

A simple example for an initial condition is~$\rho_0=\rho_0^\theta:=\frac 12 \delta_{0}+\frac 12 \delta_{\theta}$, which satisfies~$\mathrm{supp}\,\rho_0\subset [0,\pi)$ if~$\theta\in[0,\pi)$ (resp.~$\mathrm{supp}\,\rho_0\subset (\pi,2\pi)\cup \{0\}$ if~$\theta\in(\pi,2\pi)$). Then, thanks to the argument above and to the forthcoming Proposition~\ref{prop-flat-middle}, 
\begin{gather*}
 W_2(\rho_t^\theta,\delta_{\theta/2})\leqslant \pi e^{-t/4}\quad \text{if}\quad \theta\in[0,\pi),\\
 W_2(\rho_t^\theta,\delta_{\theta/2+\pi})\leqslant \pi e^{-t/4}\quad\text{if}\quad\theta\in(\pi,2\pi).
\end{gather*}
In particular, if~$\theta_1<\pi<\theta_2$, then
\[W_2(\rho_t^{\theta_1},\rho_t^{\theta_2})\geqslant W_2(\delta_{\theta_1/2},\delta_{\theta_2/2+\pi})-2\pi e^{-t/4}\geqslant \pi\left(1-2e^{-t/4}\right)\geqslant \frac{\pi}2,\]
for~$t\geqslant4\ln 4$, while~$W_2(\rho_0^{\theta_1},\rho_0^{\theta_2})=\frac{|\theta_1-\theta_2|}4$. This example shows that the measure solutions of the midpoint model on~$\mathcal M=\mathbb{S}^1$ can not be continuous with respect to the initial data.
\end{exa}

In~\cite{hinow2009analysis}, the existence and uniqueness of solutions for the midpoint model has been shown for~$\mathcal M=\mathbb{R}$ (their result indeed holds for more general models), when the initial condition is integrable, that is~$\rho_0\in L^1(\mathbb{R})$. The proof is then based on~$L^1$ estimates, different from the one used to prove Proposition~\ref{prop-existence-uniqueness}. The extension of this approach for~$\mathcal M$ a manifold ($\mathcal M=\mathbb{S}^n$, for instance) is an open problem, that we will not try to address in the present article.

\subsection{The midpoint model in~$\mathbb{R}^n$, and difficulties appearing for more general position sets}\label{subsection:Rn}

Let consider the case of~$\mathcal M=\mathbb{R}^n$ where, when two particles interact, they stick together, in the middle of the segment connecting the two particles. Under this assumption, we show that for any initial condition with finite second moment, the solution of~\eqref{eq-kinetic-model} converges to a Dirac mass as time goes to infinity. This model has already been studied in the one-dimensional case, as a simple economic market model~\cite{pareschi2006selfsimilarity}, or a model for the exchange of proteins between bacteria~\cite{hinow2009analysis}, and the use of the Wasserstein distance~$W_2$ is not new for this study~\cite{carrillo2007contractive}. We will recall below how this dynamics can be described, and discuss the difficulties arising from a similar model on a more complicated manifold (a sphere in particular).

\begin{prop}\label{prop-flat-middle}
We suppose that the collision kernel is the midpoint kernel, that is
\[K(\cdot,x_*,x_*'):=\delta_{\frac12(x_*+x_*')}.\]
For~$\rho_0$ in~$\mathcal P_2(\mathbb{R}^n)$, there exists a unique solution~$\rho$ of~\eqref{eq-kinetic-model} in~$C(\mathbb{R}_+,\mathcal P_2(\mathbb{R}^n))$. Moreover, the center of mass~$\bar x=\int_{\mathbb{R}^n}x\,d\rho(x)$ is conserved by the equation, and we have
\[W_2(\delta_{\bar x},\rho_t)=W_2(\delta_{\bar x},\rho_0)\,e^{-t/4}.\]
\end{prop}
\begin{proof}[Proof of Proposition~\ref{prop-flat-middle}]
The existence and uniqueness of a solution comes from Proposition~\ref{prop-existence-uniqueness}. Indeed, using the fact that~$W_2(\delta_x,\delta_y)=|x-y|$, we get that the map~$x_*'\mapsto\delta_{\frac12(x_*+x_*')}$ is~$\frac12$-Lipschitz for any~$x_*$. Now a simple computation gives
\begin{equation}
 \frac d{dt}\int_{\mathbb{R}^n}x\,d\rho_t(x) =\int_{\mathbb{R}^n\times\mathbb{R}^n} \bigg(\int_{\mathbb{R}^n}x\,K(x,x_*,x_*')\,dx-\frac{x_*+x_*'}2\bigg)\,d\rho_t(x_*)\,d\rho_t(x_*')=0,\label{center_mass_RN}
\end{equation}
which gives the conservation of the center of mass~$\bar x$.
We compute next the evolution of the second moment~$m_2(t):=\int_{\mathbb{R}^n}|x-\bar x|^2\,d\rho_t(x)$. We get
\begin{align}
 \frac {d\,m_2}{dt}(t)
 &=\int_{\mathbb{R}^n\times\mathbb{R}^n} \bigg(\int_{\mathbb{R}^n}|x-\bar x|^2 \,K(x,x_*,x_*')\,dx-\frac{|x_*-\bar x|^2+|x_*'-\bar x|^2}2\bigg)\,d\rho_t(x_*)\,d\rho_t(x_*')\nonumber\\
 &=\frac 12\left|\int_{\mathbb{R}^n}(x-\bar x)\,d\rho_t(x)\right|^2-\frac 12\int_{\mathbb{R}^n}|x-\bar x|^2\,d\rho_t(x)=-\frac{m_2(t)}2,\label{calcul_m2}
\end{align}
thanks to the definition of the center of mass~$\bar x$. It follows that~$m_2(t)=m_2(0)e^{-t/2}$, which ends the proof, since~$m_2$ is nothing else than~$W_2^2(\delta_{\bar x},\rho)$ thanks to~\eqref{eq-W2-dirac}.
\end{proof}

The conservation of the center of mass~$\bar x$ by~\eqref{eq-kinetic-model} plays a central role in the above computation (see~\eqref{calcul_m2}). Unfortunately, it seems difficult to define an equivalent of the center of mass for a probability distribution on the sphere, that would be conserved by the equation~\eqref{eq-kinetic-model}. For instance the direction of the first moment of the distribution is not conserved by the equation, as we show in the example below.

\begin{exa}\label{example2}
We consider the midpoint model on the circle described in Example~\ref{example1}. Let~$\rho$ be a solution in~$C(\mathbb{R}_+,\mathcal P(\mathbb{S}^1))$ with initial condition~$\rho_0=\frac 13\sum_{k=1}^3\delta_{e^{i\theta_k}}$ with~$\theta_1=0$,~$\theta_2=\pi-2\varepsilon$, and~$\theta_3=\pi+\varepsilon$. Then, one can compute that
\[\int_{\mathbb{S}^1} x\rho(0,x)\,dx=\left(\frac{-1}3,\frac{\varepsilon}3\right)+\mathcal O(\varepsilon^2),\]
\begin{eqnarray*}
 \frac{d}{dt}\Big|_{t=0}\int_{\mathbb{S}^1} x\rho(t,x)\,dx&=&\int_{\mathbb{S}^1} x\left(\frac{-2}3\rho_0+\frac 29\left(\delta_{x_{12}}+\delta_{x_{13}}+\delta_{x_{23}}\right)\right)\,dx\\
&=&\left(\frac{-4}{9}+\frac{\varepsilon}3,\frac{-\varepsilon}{9}\right)+\mathcal O(\varepsilon^2),
\end{eqnarray*}
where~$x_{12}=e^{i(\pi/2-\varepsilon)}$,~$x_{13}=e^{i(3\pi/2+\varepsilon/2)}$ and~$x_{23}=e^{i(\pi-\varepsilon/2)}$. When~$\varepsilon>0$ is small enough, the vectors~$\int_{\mathbb{S}^1} x\rho(0,x)\,dx$ and~$\frac{d}{dt}|_{t=0}\int_{\mathbb{S}^1} x\rho(t,x)\,dx$ are then not collinear, which implies that neither the vector~$y_0:=\int_{\mathbb{S}^1} x\rho(0,x)\,dx$ nor its direction~$\frac{y_0}{\left\|y_0\right\|}$ can be conserved by the midpoint model on the circle.

In Figure~\ref{fig-simu1}, we present a numerical simulation of the model described in Example~\ref{example1}. Numerically, we also observe that the direction of the center of mass of the solution is not conserved by the equation.
\end{exa}
\begin{figure}[h]
\centering
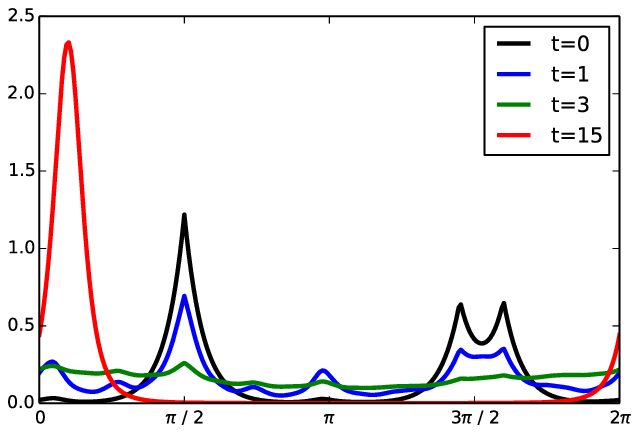\hfill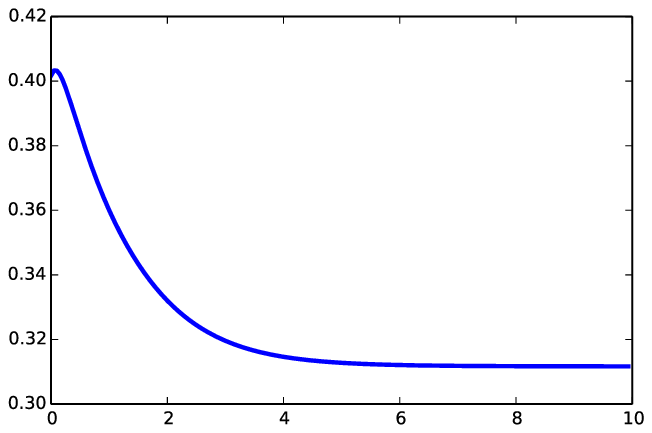
\caption{Numerical simulation of the midpoint model on the circle described in Example~\ref{example1} (colors online). On the left, we have represented the solution~$x\mapsto\rho(t,x)$ at various times, as a function of the argument of~$x$, that is~$\theta=\mathrm{arg}(x)\in [0,2\pi)$. On the right, we have represented the argument~$\theta_1(t)$ of the center of mass~$\int_{x\in\mathbb{S}^1\subset \mathbb{R}^2} x\rho(t,x)\,dx$, as a function of the time~$t$. This numerical simulation is based on a finite difference scheme.}
\label{fig-simu1} 
\end{figure}

To circumvent this difficulty, we notice that a computation close to~\eqref{calcul_m2} can be performed based on a variant of~$m_2$, defined without the center of mass, namely
\begin{equation}
\label{eq-tilde-m2}\widetilde m_2:=\int_{\mathbb{R}^n\times\mathbb{R}^n}|x-y|^2\,d\rho_t(x)\, d\rho_t(y).
\end{equation}
This quantity, just as~$m_2$, is a Lyapunov functional for~\eqref{eq-kinetic-model} for the midpoint model when~$\mathcal M=\mathbb{R}^n$ (it is actually equal to~$2m_2$, thanks to the conservation of mass and center of mass by the midpoint model in~$\mathbb{R}^n$). When~$\mathcal M$ is a Riemannian manifold, we will then use the counterpart of this quantity~$\widetilde m_2$, namely~\eqref{eq-tilde-m2} where~$|x-y|$ is replaced by the geodesic distance~$d(x,y)$. We do not know if this quantity is also a Lyapunov functional for general positions sets~$\mathcal M$ (see Figure~\ref{fig-simu2}), but we will show however that this quantity is a Lyapunov functional in a neighborhood of any Dirac mass on~$\mathcal M$ (the neighborhood of a Dirac mass can be defined thanks to the Wasserstein distance~$W_2$). The latter property is the root of our analysis.

\section{The midpoint model on the sphere}\label{section:midpoint-sphere}

We now consider the case where~$\mathcal M$ is the unit sphere of~$\mathbb{R}^{n+1}$, that we will denote by~$\mathbb{S}$ in this section (instead of~$\mathbb{S}^n$), to simplify notations. Notice that indeed, the dimension of~$\mathcal M$ will play little role in our analysis.

In this section, we will assume that when two particles with directions~$x_*$ and~$x_*'$ collide, they align: their new direction is the mean direction of~$x_*$ and~$x_*'$. This assumption holds unless those two directions are antipodal, and in that case, their new direction should belong to the corresponding equator, see Figure~\ref{fig-midpoint-sphere}. We are thus interested in the solutions of the kinetic equation~\eqref{eq-kinetic-model}, where the probability law~$K(\cdot,x_*,x_*')$ satisfies the following assumption:
\begin{ass}[Midpoint models on the sphere~$\mathbb{S}$]\label{Ass:midpoint_sphere}For any~$x_*,\,x_*'\in \mathbb{S}^2$,~$K(\cdot,x_*,x_*')$ is a probability measure, and:
 \begin{itemize}
\item if~$x_*\neq-x_*'$, then~$K(\cdot,x_*,x_*')=\delta_{x_m}$, where~$x_m=\frac{x_*+x_*'}{\|x_*+x_*'\|}$,
\item if~$x_*=-x_*'$, then the support of~$K(\cdot,x_*,x_*')$ is included in~$\{x\in\mathbb{S},x\cdot x_*=0\}$.
\end{itemize}
\end{ass}
%
\begin{figure}[h]
\centering
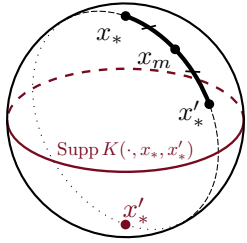
\caption{Two types of collision on the sphere.}
\label{fig-midpoint-sphere} 
\end{figure}

We will denote by~$d(x,y)$ the geodesic distance on~$\mathbb{S}$ between~$x$ and~$y$, given by the length of the arc of a great circle passing through~$x$ and~$y$: we get
\begin{equation}\label{acos}
 d(x,y)=\arccos(x\cdot y). 
\end{equation}

The main result concerns the nonlinear stability of Dirac masses.
\begin{theo}
\label{theo-stability-dirac-sphere}
Let~$K$ satisfy Assumption~\ref{Ass:midpoint_sphere}. There exists~$C_1>0$ and~$\eta>0$ such that for any solution~$\rho\in C(\mathbb{R}_+,\mathcal P(\mathbb{S}))$ of~\eqref{eq-kinetic-model} with initial condition~$\rho_0$ satisfying~$W_2(\rho_0,\delta_{x_0})< \eta$ for some~$x_0\in\mathbb{S}$, there exists~$x_\infty\in \mathbb{S}$ such that
\[W_2(\rho_t,\delta_{x_\infty})\leqslant C_1  W_2(\rho_0,\delta_{x_0}) \, e^{-t/4}.\]
\end{theo}

\begin{rem}\label{rem:convergence_rate}
 The convergence rate is the same as in the midpoint model in~$\mathbb{R}^n$ (see Proposition~\ref{prop-flat-middle}). This property will also hold for midpoint models on more general manifolds, see Theorem~\ref{theo-stability-dirac-midpoint-M}. This convergence rate will however be different if the collision kernel is less contracting than the midpoint model, see Section~\ref{section:non-midpoint}.

This convergence rate is optimal. Indeed, if we consider an initial condition supported on a geodesic~$\gamma:[0,1]\to\mathcal M$ connecting two points~$x_1,\,x_2\in\mathcal M$ such that~$d(x_1,x_2)>0$ is small, then for any~$\tilde x_1,\,\tilde x_2\in \gamma([0,1])$, the unique minimal geodesic connecting~$\tilde x_1$ to~$\tilde x_2$ is~$\gamma|_{[\tilde x_1,\tilde x_2]}$, and in particular, the midpoint of this geodesic is included in~$\gamma([0,1])$. The support of~$\rho(t,\cdot)$ is then included in~$\gamma([0,1])$ at all times, and~$\gamma^{-1}_*\rho(t,\cdot)$ is a solution of the midpoint model on~$\mathbb{R}$. The convergence rate of~$t\mapsto\gamma^{-1}_*\rho(t,\cdot)$ to a Dirac mass is then explicitly given by Proposition~\ref{prop-flat-middle}, and is optimal since the exponential decay is an equality. The convergence rate of~$t\mapsto \rho(t,\cdot)=\gamma_*\left(\gamma^{-1}_*\rho(t,\cdot)\right)$ to a Dirac mass is then the same, which shows the optimality of the convergence rate obtained in Theorem~\ref{theo-stability-dirac-sphere}. Notice that this 
argument indeed holds in the more general context where the position set~$\mathcal M$ is a Riemannian manifold (see Section~\ref{section:manifold}). Numerically, the convergence rate generically observed corresponds to the rate provided by Theorem~\ref{theo-stability-dirac-sphere}, see Figure~\ref{fig-simu2}.
\end{rem}

\begin{figure}[h]
\centering
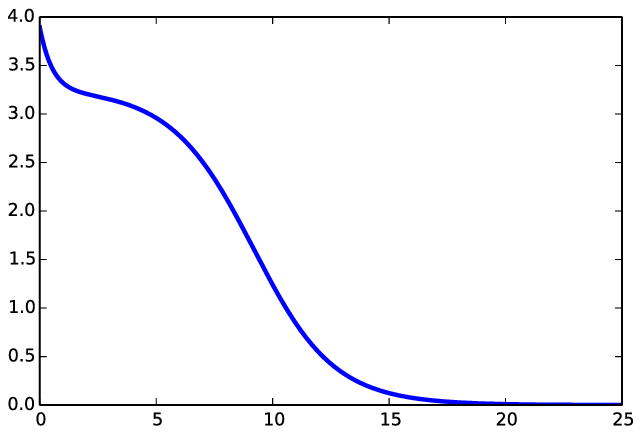\hfill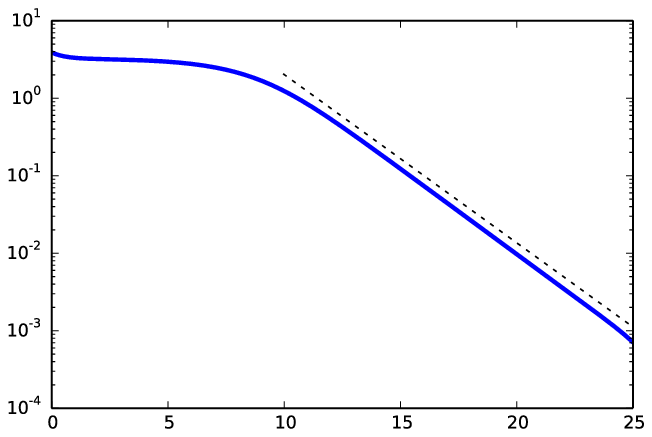
\caption{For the numerical solution presented in Figure~\ref{fig-simu1}, we represent the time evolution of the energy~$E(\rho(t,\cdot))$ (see~\eqref{def-energy}), in a standard plot (left), and semi-log plot(right). We observe that the energy decreases slowly when the density is close to the uniform distribution (for~$t\sim 3$, see Figure~\ref{fig-simu1}) for which the energy is equal to~$\frac{\pi^2}3\approx3.29$. For~$t\gtrsim 10$, the energy seems to decrease exponentially as~$t\mapsto e^{-t/2}$ (represented with a dotted line), as estimated by Theorem~\ref{theo-stability-dirac-sphere}.}
\label{fig-simu2} 
\end{figure}

The remainder of this section is devoted to the proof of Theorem~\ref{theo-stability-dirac-sphere}.

\subsection{Energy and 2-Wasserstein distance}

Our analysis will be based on the following energy functional:
\begin{equation}
  E(\rho)=\int_{\mathbb{S}\times\mathbb{S}} d(x,y)^2\,d\rho(x)\,d\rho(y),
\label{def-energy}
\end{equation}
First of all, we give a technical lemma that will be used later in the paper. It concerns the link between this energy and the Wasserstein distance~$W_2$ on~$\mathcal P(\mathbb{S})$, and provides some useful Markov-type inequalities. 
\begin{lem}\label{lem-cheb-sphere}
If~$\rho\in\mathcal P(\mathbb{S})$, we have for any~$x\in\mathbb{S}$,
\begin{equation}
E(\rho)\leqslant 4\,W_2(\rho,\delta_{x})^2,\label{E-4W2}
\end{equation}
and there exists~$\bar x\in \mathbb{S}$ such that
\begin{equation}
W_2(\rho,\delta_{\bar x})^2\leqslant E(\rho).\label{Wxbar}
\end{equation}
In that case, for any~$\kappa>0$, we have
\[\int_{\{x\in \mathbb{S};\,d(x,\bar x)\geqslant \kappa\}}d\rho(x)\leqslant \frac 1{\kappa^2} E(\rho),\quad \text{and}\quad \int_{\{x\in \mathbb{S};\,d(x,\bar x)\geqslant \kappa\}}d(x,\bar x)\,d\rho(x)\leqslant \frac 1{\kappa} E(\rho).\]

\end{lem}

\begin{proof}[Proof of Lemma~\ref{lem-cheb-sphere}]
Let~$\rho\in\mathcal P(\mathbb{S})$. Using~\eqref{eq-W2-dirac}, we get
\begin{align*}
E(\rho)&=\int_{\mathbb{S}}W_2(\rho,\delta_y)^2\,d\rho(y)  \\
&\leqslant \int_{\mathbb{S}}[W_2(\rho,\delta_{x})+d(y,x)]^2\,d\rho(y)  \\
&\leqslant\int_{\mathbb{S}}2\,[W_2(\rho,\delta_{x})^2+d(y,x)^2]\,d\rho(y) =4\, W_2(\rho,\delta_{x})^2,
\end{align*}
which gives the estimate~\eqref{E-4W2}.
The function~$y\mapsto W_2(\rho,\delta_{y})$ is continuous (1-Lipschitz), and reaches its minimum at a point~$\bar x\in\mathbb{S}$.
This gives
\[E(\rho)=\int_{\mathbb{S}} W_2(\rho,\delta_{y})^2d\rho(y) \geqslant \int_{\mathbb{S}} W_2(\rho,\delta_{\bar x})^2d\rho(y)=W_2(\rho,\delta_{\bar x})^2.\]
Now, for~$\kappa>0$, we have
\[\int d(x,\bar x)^2\,d\rho(x)\geqslant \kappa\int_{\{x\in \mathbb{S};\,d(x,\bar x)\geqslant \kappa\}} d(x,\bar x) d\rho(x) \geqslant \kappa^2\int_{\{x\in \mathbb{S};\,d(x,\bar x)\geqslant \kappa\}} d\rho(x),\]
which, combined to the first inequality, proves the lemma.
\end{proof}

\subsection{Energy estimates on collisions}
Here, we show how collisions decrease the energy with the right rate.
These estimates will be the main argument to show the local stability of Dirac masses. 

First of all, let us compute the derivative of the energy, when~$\rho\in C(\mathbb{R}_+,\mathcal P(\mathbb{S}))$ is a solution of~\eqref{eq-kinetic-model}:
\begin{align}
 \frac12\frac d{dt}E(\rho)&=\int_{\mathbb{S}\times\mathbb{S}} d(x,y)^2\partial_t\rho(x)\rho(y)\,dx\,dy\nonumber\\
&=\int_{\mathbb{S}\times\mathbb{S}}d(x,y)^2\rho(y)[A(\rho,\rho)(x)-\rho(x)]\,dx\,dy.\label{dE1}
\end{align}
Since~$\rho$ is a probability measure, we can write
\[ \int_{\mathbb{S}} d(x,y)^2\,d\rho(x)=\int_{\mathbb{S}\times\mathbb{S}} \frac{d(x_*,y)^2+d(x_*',y)^2}2 \,d\rho(x_*)\,d\rho(x_*'),\]
and we insert this expression into~\eqref{dE1}. Together with the definition~\eqref{def-A} of~$A$, we obtain:
\begin{equation}
 \frac12\frac d{dt}E(\rho)=\int_{\mathbb{S}\times\mathbb{S}\times\mathbb{S}} \alpha(x_*,x_*',y) \,d\rho(x_*)\,d\rho(x_*')\,d\rho(y).\label{dEdt}
\end{equation}
where the function~$\alpha$ is defined by
\begin{equation}
\alpha(x_*,x_*',y)=\int_{\mathbb{S}} d(x,y)^2 K(x,x_*,x_*')\,dx-\frac{d(x_*,y)^2+d(x_*',y)^2}2.\label{def-alpha}
\end{equation}
This quantity~$\alpha(x_*,x_*',y)$ corresponds to the variation of the energy caused by the collision of two particles located at~$x_*$ and~$x_*'$, and obtaining good estimates on this quantity will be the main ingredient to prove the decay of the energy. We will need two precise estimates of this quantity~$\alpha$. The first one is a global estimate and is simply a consequence of triangle inequalities.

\begin{lem}\label{lem:E-particles-global}
Let~$K$ satisfy Assumption~\ref{Ass:midpoint_sphere}. For any~$x_*,\,x_*',\,y\in \mathbb{S}$, we have 
\begin{equation}
\alpha(x_*,x_*',y)\leqslant -\frac14 d(x_*,x_*')^2 + 2\,d(x_*,x_*') \min\big(d(x_*,y),\,d(x_*',y)\big).\label{estimate-d1}
\end{equation}
\end{lem}

\begin{proof}[Proof of Lemma~\ref{lem:E-particles-global}]
The triangle inequality in the triangle~$(x_*,x,y)$ gives
\[d(x,y)^2\leqslant d(x_*,y)^2+2\,d(x_*,y)\,d(x,x_*)+d(x,x_*)^2,\]
which, after integrating against the probability measure~$K(\cdot,x_*,x_*')$ and using the fact that for any~$x$ in the support of~$K(\cdot,x_*,x_*')$, we have~$d(x,x_*)=\frac12d(x_*,x_*')$, gives
\[\int_{\mathbb{S}} d(x,y)^2 K(x,x_*,x_*')\,dx\leqslant d(x_*,y)^2+\,d(x_*,y)\,d(x_*,x_*')+\tfrac14\,d(x_*,x_*')^2.\]
We now write the triangle inequality in the triangle~$(x_*,x_*',y)$ to get
\[\tfrac12d(x_*',y)^2\geqslant\tfrac12d(x_*,y)^2-\,d(x_*,y)\,d(x_*',x_*)+\tfrac12d(x_*',x_*)^2,\]
and we insert these two last estimates in the expression~\eqref{def-alpha} of~$\alpha$. 
We obtain
\[\alpha(x_*,x_*',y)\leqslant-\tfrac14d(x_*',x_*)^2+2\,d(x_*,y)\,d(x_*',x_*).\]
Since~$x_*$ and~$x_*'$ play symmetric roles, this ends the proof of the Lemma.
\end{proof}

The second estimate is a local estimate of the quantity~$\alpha(x_*,x_*',y)$ when the points~$x_*$,~$x_*'$ and~$y$ are sufficiently close.
\begin{lem}\label{lem:E-particles-local}
Let~$K$ satisfy Assumption~\ref{Ass:midpoint_sphere}. For any~$\kappa_1<\frac{2\pi}3$, there exists~$C_1\geqslant0$ such that for any~$\kappa\leqslant\kappa_1$, for any~$x_*,\,x_*',\,y\in \mathbb{S}$ such that
$\max\left(d(x_*,y),\,d(x_*',y),\,d(x_*,x_*')\right)\leqslant \kappa$, we have
\[\alpha(x_*,x_*',y) \leqslant -\frac14d(x_*,x_*')^2+ C_1\, \kappa^2\,d(x_*,x_*')^2.\]
\end{lem}
\begin{proof}[Proof of Lemma~\ref{lem:E-particles-local}]
Let~$x_m$ be the midpoint of a minimal geodesic joining~$x_*$ and~$x_*'$. To simplify notations, we will write~$a=d(x_m,x_*)=d(x_m,x_*')=\frac12d(x_*,x_*')$,~$b=d(x_*,y)$, and~$b'=d(x_*',y)$ (see figure~\ref{fig-apollonius-sphere}). We also denote by~$m=d(x_m,y)$ the length of the median. We recall the formula~\eqref{acos} for the calculation of distances.
\begin{figure}[h]
\centering
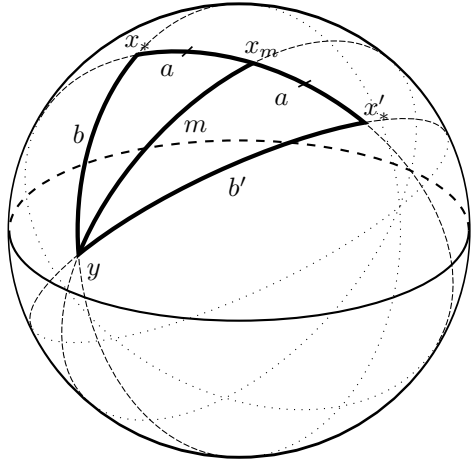
\caption{A spherical triangle and its median}
\label{fig-apollonius-sphere} 
\end{figure}

We first give the equivalent of Apollonius theorem in spherical geometry:
\begin{equation}
\frac{\cos b + \cos b'}2=\cos a \cos m.\label{apollonius-spherical}
\end{equation}
To prove this relation, we have~$\|x_*+x_*'\|^2=2+2\,x_*\cdot x_*'=2\,(1+\cos 2a)=4 \cos^2 a$ and therefore, since~$a\leqslant\frac{\pi}2$, we get~$\|x_*+x_*'\|=2\cos a$. When~$x_*+x_*'\neq0$, we get
\[\cos m=y\cdot x_m=y\cdot\frac{x_*+x_*'}{\|x_*+x_*'\|}=\frac{\cos b+\cos b'}{2\cos a}.\]
When~$x_*=-x_*'$, we have~$\cos b=x_*\cdot y=-x_*'\cdot y=-\cos b'$, and we have~$2a=d(x_*,x_*')=\pi$, so~$\cos a=0$ and the formula~\eqref{apollonius-spherical} is still valid.

We now introduce the function~$g:t\in[-1,1]\mapsto\arccos(t)^2$, such that~$g(\cos \theta)=\theta^2$.
We obtain~$g'(\cos \theta)=-\frac{2\theta}{\sin\theta}$, which gives that~$g'$ is negative and that~$\theta\mapsto g'(\cos \theta)$ is decreasing on~$(0,\pi)$.
Therefore~$g'$ is increasing on~$(-1,1)$, so the function~$g$ is decreasing and convex on~$[-1,1]$.
Thus we have
\[g(\cos a \cos m)=g\left(\frac{\cos b+\cos b'}2\right)\leqslant\frac{g(\cos b)+g(\cos b')}2=\frac{b^2+b'^2}2\]
But we also have that~$g$ is above its tangent, so we get
\begin{align*}
g(\cos a \cos m)&=g(\cos m - (1-\cos a)\cos m)\\
&\geqslant g(\cos m) - g'(\cos m)(1-\cos a)\cos m\\
&\geqslant m^2+\frac{2\,m\cos m}{\sin m} (1-\cos a).
\end{align*}
Combining these two inequalities, we obtain
\begin{equation}
m^2-\frac{b^2+b'^2}2\leqslant - (1-\cos a)\,\frac{2\,m \cos m}{\sin m}\label{apollonius-estimate-cos}.
\end{equation}

We fix~$0<\kappa\leqslant\kappa_1<\frac{2\pi}3$ and we suppose that~$b,b',2a\leqslant\kappa$.
By triangle inequality in the geodesic triangle~$(x_m,x_*,y)$, we get that~$m\leqslant\frac32\kappa$.
We have~$\frac{m\cos m}{\sin m}=1+O(m^2)$ (this function is even and~$C^2$ on~$[-\frac32\kappa_1,\frac32\kappa_1]$ since~$\frac32\kappa_1<\pi$) and~$2(1-\cos a)=a^2(1+O(a^2))$.
Therefore, there exists~$C>0$ such that
\begin{equation}\label{appo}
 m^2-\frac{b^2+b'^2}2\leqslant a^2(-1+C\,\kappa^2),
\end{equation}
To conclude, we notice that since~$\kappa<\pi$, the geodesic from~$x_*$ to~$x_*'$ is unique and~$K(\cdot,x_*,x_*')=\delta_{x_m}$, and then~$\alpha(x_*,x_*',y)=m^2-\frac12(b^2+b'^2)$.
\end{proof}

\begin{rem}
When~$\mathbb{S}=\mathcal S^1\subset \mathbb{R}^2$ is the unit circle, then, if the three lengths~$d(x_*,x_*')$,~$d(x_*,y)$ and~$d(x_*',y)$ are less than~$\frac{2\pi}3$, the geodesic triangle~$(x_*,x_*',y)$ cannot be the whole circle. The computation is then the same as it would be in the real line, and we directly get the estimate of Lemma~\ref{lem:E-particles-local} with~$C_1=0$.
\end{rem}

\subsection{Energy contraction property}

In this subsection, we will prove the main argument of the proof of Theorem~\ref{theo-stability-dirac-sphere}, which is an energy contraction property showing that the energy~$E(\rho_t)$ converges exponentially fast to~$0$, provided it is initially sufficiently small. A similar estimate will also be the key argument in more general situations, see Sections~\ref{section:manifold} and~\ref{section:non-midpoint}.
\begin{prop}\label{prop-decay-energy}
Let~$K$ satisfy Assumption~\ref{Ass:midpoint_sphere}. There exists~$C_0>0$,~$M_0>0$ and~$\gamma>0$ such that for any solution~$\rho\in C(\mathbb{R}_+,\mathcal P(\mathbb{S}))$ of~\eqref{eq-kinetic-model}, we have, for any~$t\in\mathbb{R}_+$ such that~$E(\rho_t)\leqslant M_0$:
\begin{equation}
\label{eq-energy-contraction}
 \frac12\frac d{dt}E(\rho)\leqslant -\frac14E(\rho)+ C_0\, E(\rho)^{1+\gamma}.
\end{equation}
Consequently, there exist~$E_{\max}>0$ such that any solution~$\rho\in C(\mathbb{R}_+,\mathcal P(\mathbb{S}))$ of~\eqref{eq-kinetic-model} with initial condition~$\rho_0$ such that~$E(\rho_0)< E_{\max}$ satisfies
\begin{equation}
\label{eq-decay-energy}
E(\rho_t)\leqslant \Big[1-\Big(\frac{E(\rho_0)}{E_{\max}}\Big)^{\gamma}\Big]^{-1/\gamma} E(\rho_0) \,e^{-\frac12t}.
\end{equation}
\end{prop}

\begin{proof}[Proof of Proposition~\ref{prop-decay-energy}]
Let~$\rho\in C(\mathbb{R}_+,\mathcal P(\mathbb{S}))$ be a solution of~\eqref{eq-kinetic-model}.
We have, thanks to~\eqref{dEdt}, and the fact that~$\rho$ is a probability measure
\begin{equation}
 \frac12\frac d{dt}E(\rho)+\frac14E(\rho) =\int_{\mathbb{S}\times\mathbb{S}\times\mathbb{S}} [\alpha(x_*,x_*',y)+\tfrac14d(x_*,x_*')^2]\,d\rho(x_*)\,d\rho(x_*')\,d\rho(y).\label{eqE}
\end{equation}
We want to estimate the right-hand side of~\eqref{eqE} for a given time~$t$.
We take~$\kappa\leqslant\kappa_1$ and~$\bar x$ given by Lemma~\ref{lem-cheb-sphere}. We then define~$\bar{\omega}:=\{x\in \mathbb{S};\,d(x,\bar x)\leqslant \kappa/2\}$, and we split the triple integral depending on whether~$x_*$,~$x_*'$ and~$y$ are in~$\bar{\omega}$ or not.
When both of them are in~$\bar{\omega}$, we estimate the triple integral thanks to Lemma~\ref{lem:E-particles-local}, since~$\max(d(x_*,x_*'),d(x_*,y),d(x_*',y))\leqslant\kappa$.
When one of them is not in~$\bar{\omega}$, we estimate the triple integral thanks to Lemma~\ref{lem:E-particles-global}.
We obtain
\begin{align*}
\frac12\frac d{dt}E(\rho)+\frac14E(\rho)&\leqslant C\,\kappa^2\,\iiint\limits_{x_*,\,x_*',\,y\in\bar{\omega}}d(x_*,x_*')^2\,d\rho(x_*)\,d\rho(x_*')\,d\rho(y)\\
&\phantom{\leqslant}+2\iiint\limits_{y\in\bar{\omega}^c,\,x_*,\,x_*'\in\mathbb{S}} \min(d(x_*,y),d(x_*',y))\,d(x_*,x_*')\,d\rho(x_*)\,d\rho(x_*') \,d\rho(y)\\
&\phantom{\leqslant}+2\iiint\limits_{x_*\in\bar{\omega}^c,\,x_*'\in\mathbb{S},\,y\in\bar{\omega}} \min(d(x_*,y),d(x_*',y))\,d(x_*,x_*')\,d\rho(x_*')\,d\rho(y)\,d\rho(x_*)\\
&\phantom{\leqslant}+2\iiint\limits_{x_*'\in\bar{\omega}^c,\,x_*,\,y\in\bar{\omega}} \min(d(x_*,y),d(x_*',y))\,d(x_*,x_*')\,d\rho(x_*)\,d\rho(y)\,d\rho(x_*').\\
\end{align*}
The first triple integral is directly estimated by~$E(\rho)$ since~$\rho$ is a probability measure.
To control the three other integrals, we exchange the role of~$x_*$,~$x_*'$ and~$y$, and extend the domain of integration for the double integral, and we see that they can all be estimated similarly by~$J(\rho)$, described below. We get
\[\frac12\frac d{dt}E(\rho)+\frac14E(\rho)\leqslant C\, \kappa^2 \,E(\rho) + 6 J(\rho),\]
where
\begin{align*}
J(\rho)&=\iiint\limits_{x\in\bar{\omega}^c,y,z\in\mathbb{S}}d(x,y)\,d(y,z)\,d\rho(y)\,d\rho(z)\,d\rho(x)\\
&\leqslant\iiint\limits_{x\in\bar{\omega}^c,y,z\in\mathbb{S}}[d(x,\bar x)+d(\bar x,y)]\,d(y,z)\,d\rho(y)\,d\rho(z)\,d\rho(x)\\
&\leqslant\int\limits_{x\in\bar{\omega}^c}d(x,\bar x)\,d\rho(x) \iint\limits_{y,z\in\mathbb{S}}d(y,z)\,d\rho(y)\,d\rho(z)+\int\limits_{x\in\bar{\omega}^c}d\rho(x) \iint\limits_{y,z\in\mathbb{S}}d(y,\bar x)\,d(y,z)\,d\rho(y)\,d\rho(z)\\
&\leqslant\frac{2E(\rho)}{\kappa} \iint_{\mathbb{S}\times\mathbb{S}}d(y,z)\,d\rho(y)\,d\rho(z)+\frac{4E(\rho)}{\kappa^2} \iint_{\mathbb{S}\times\mathbb{S}}d(y,\bar x)\,d(y,z)\,d\rho(y)\,d\rho(z)
\end{align*}
thanks to Lemma~\ref{lem-cheb-sphere} (with~$\kappa/2$).
The two double integrals can then be estimated thanks to the Cauchy-Schwarz inequality (the first one is less than~$\sqrt{E(\rho)}$ since~$\rho\otimes\rho$ is a probability measure, and the second one is less than~$E(\rho)$ thanks to~Lemma~\ref{lem-cheb-sphere}).
In the end, we get, as soon as~$\kappa\leqslant\kappa_1$:
\begin{equation}
\frac12\frac d{dt}E(\rho)+\frac14E(\rho)\leqslant C\,\kappa^2\,E(\rho)+12\frac{E(\rho)^{\frac32}}{\kappa}+24\frac{E(\rho)^2}{\kappa^2}.\label{estimate-dE-kappa}
\end{equation}
So we can take~$\kappa=E(\rho)^{\frac16}$, and as soon as~$E(\rho)\leqslant M_0$, we get the estimate~\eqref{eq-energy-contraction}, with~$M_0=(\kappa_1)^6$,~$\gamma=\frac13$, and~$C_0=C+12+24M_0^{\frac13}$.

Now, writing~$E_{\max}=\min((4C_0)^{-1/\gamma},M_0)$ and~$z(t)=\frac{E(\rho)}{E_{\max}}$, the estimate~\eqref{eq-energy-contraction} becomes
\[\frac{dz}{dt}\leqslant-\tfrac12 \,z\,(1-z^\gamma),\] which is then satisfied as soon as~$z<1$, since~$E_{\max}\leqslant M_0$, and can be explicitly solved.
We get that, as soon as~$z(0)<1$,~$z$ is decreasing in time and we have
\[z(t)\leqslant\frac{z(0)e^{-\frac12t}}{\big[1-z(0)^{\gamma}(1-e^{-\frac12\gamma t})\big]^{1/\gamma}}\leqslant\frac{z(0)e^{-\frac12t}}{\big(1-z(0)^{\gamma}\big)^{1/\gamma}},\]
which gives the estimate~\eqref{eq-decay-energy}.
\end{proof}

\subsection{Proof of Theorem~\ref{theo-stability-dirac-sphere}}\label{proof-shperes}
\label{subsec-local-stability}
We are now ready to complete the proof of the main result of this section.
\begin{proof}[Proof of Theorem~\ref{theo-stability-dirac-sphere}]
If~$W_2(\rho_0,\delta_{x_0})$ is sufficiently small for some~$\bar x\in\mathbb{S}$, then~$E(\rho_0)=E_0$ is also sufficiently small, since~$E(\rho_0)\leqslant4W_2(\rho_0,\delta_{x_0})^2$ thanks to~\eqref{E-4W2}. We can apply Proposition~\ref{prop-decay-energy}. In what follows, we will denote by~$C$ a generic constant, which may vary from line to line, but which does not depend on~$E_0$ (provided~$W_2(\rho_0,\delta_{x_0})$ is sufficiently small for some~$\bar x\in\mathbb{S}$). Our goal is to show that~$\bar x(t)$ (provided by Lemma~\ref{lem-cheb-sphere}) satisfies a Cauchy property.
Let~$t'\geqslant t$,
\begin{align}
d(\bar x(t),\bar x(t'))^2&=\int d(\bar x(t),\bar x(t'))^2\,d\rho_{t'}(x)\leqslant \int 2\big[d(\bar x(t),x)^2+d(x,\bar x(t'))^2\big]\,d\rho_{t'}(x) \nonumber\\
&\leqslant2\int d(\bar x(t),x)^2d\rho_{t'}(x)+2\int d(x,\bar x(t'))^2\,d\rho_{t'}(x)\nonumber\\
&\leqslant2\int d(x,\bar x(t))^2\,d\rho_{t}(x)+2\int_t^{t'}\frac d{ds}\int d(x,\bar x(t))^2\,d\rho_{s}(x)\,ds+C\,E_0\,e^{-\frac12t'}\nonumber\\
&\leqslant C\,E_0\,[e^{-\frac12t}+e^{-\frac12t'}]+2\int_t^{t'}\frac d{ds}\int d(x,\bar x(t))^2\,d\rho_{s}(x)\,ds.\label{estdxx}
\end{align}
To estimate the second term of~\eqref{estdxx}, we use once again a computation similar to the one introduced in the beginning of this section (see~\eqref{dE1},~\eqref{dEdt},~\eqref{def-alpha}):
\begin{align*}
\frac d{ds}\int &d(x,\bar x(t))^2\,d\rho_{s}(x)=\int d(x,\bar x(t))^2\partial_s\rho(s,x)\,dx\\
&=\int d(x,\bar x(t))^2[A(\rho_s,\rho_s)(x)-\rho(s,x)]\,dx\\
&=\iint \alpha(x_*,x_*',\bar x(t))\,d\rho_{s}(x_*)d\rho_s(x_*'),
\end{align*}
and then, thanks to Lemma~\ref{lem:E-particles-global},
\begin{align}
 \frac d{ds}\int d(x,\bar x(t))^2\,d\rho_{s}(x)&\leqslant2\iint d(x_*,\bar x(t))\, d(x_*,x_*')\,d\rho_s(x_*)\,d\rho_s(x_*')\nonumber\\
 &\leqslant2\iint d(x_*,\bar x(s))d(x_*,x_*')\,d\rho_s(x_*)\,d\rho_s(x_*')\nonumber\\
&\phantom{\leqslant\iint}+ 2\, d(\bar x(t),\bar x(s))\iint d(x_*,x_*')\,d\rho_s(x_*)\,d\rho_s(x_*')\nonumber\\
 &\leqslant2[E(\rho_s) + d(\bar x(t),\bar x(s))\sqrt{E(\rho_s)}]\label{est0}
\end{align}
where we have used the Cauchy-Schwarz inequality. If we define
\[M_{t,t'}=\sup_{s\in[t,t']}d(\bar x(t),\bar x(s)),\]
and plug this last estimate into~\eqref{est0}, we obtain, thanks to Proposition~\ref{prop-decay-energy}:
\begin{equation}
\frac d{ds}\int d(x,\bar x(t))^2\,d\rho_{s}(x)\leqslant C \,[E_0\,e^{-\frac12s} + \sqrt{E_0}\,M_{t,t'}\,e^{-\frac14s}].\label{estdds}
\end{equation}
If we combine the estimate~\eqref{estdds} with the estimate~\eqref{estdxx}, we get:
\begin{equation}\label{estimexpo}
d(\bar x(t),\bar x(t'))^2\leqslant C\,  \big[E_0\,e^{-\frac12t}+ \sqrt{E_0}\,M_{t,t'}\,e^{-\frac14t}\big].
\end{equation}
We then obtain
\[M_{t,t'}^2\leqslant C\,[E_0\,e^{-\frac12t} + M_{t,t'}\, \sqrt{E_0}\, e^{-\frac14t}],\]
which can be solved explicitly, and we get~$M_{t,t'}\leqslant\frac12\big[C+\sqrt{C^2+4C}\big]\sqrt{E_0}\, e^{-\frac14t}$.
Plugging back this estimate in~\eqref{estimexpo}, we now obtain
\begin{equation}
\label{estimexpo2}
d(\bar x(t),\bar x(t'))^2\leqslant C\,E_0\,e^{-\frac12t}.
\end{equation}
 
Let~$t_n:=n$. This last estimate shows that~$(\bar x(t_n))_n$ is a Cauchy sequence, and thus converges to a limit~$\bar x_\infty$.
A limit~$t'=t_n\to\infty$ of the inequality~\eqref{estimexpo2} then shows that for~$t\geqslant 0$,
\begin{equation}\label{cvxt}
d(\bar x(t),\bar x_\infty)^2\leqslant C\,E_0\,e^{-\frac12t}.
\end{equation}

We can now show the convergence of~$\rho_t$ to~$\delta_{\bar x_\infty}$, thanks to Lemma~\ref{lem-cheb-sphere} and~\eqref{cvxt}:
\begin{align*}
W_2(\rho_t,\delta_{\bar x_\infty})&\leqslant W_2(\rho_t,\delta_{\bar x(t)})+ d(\bar x_\infty,\bar x(t))\\
&\leqslant\sqrt{E(\rho_t)}+C\sqrt{E_0}\,e^{-\frac14t}\leqslant C\sqrt{E_0}\,e^{-\frac14t} \leqslant C\, W_2(\rho_0,\delta_{x_0})\,e^{-\frac14t},
\end{align*}
thanks to Proposition~\ref{prop-decay-energy} and to the estimate~\eqref{E-4W2}, which ends the proof of the Theorem.
\end{proof}

\section{The midpoint model on a Riemannian manifold}\label{section:manifold}

In this section, we consider a set of position~$\mathcal M$ that is a Riemannian manifold. More precisely, we will need the following assumptions on~$\mathcal M$ (which will be needed for comparison arguments, we refer to~\cite{berger2003panoramic} for the definition of injectivity radius and sectional curvature):

\begin{ass}[Position set~$\mathcal M$]\label{Ass-M}
$\mathcal M$ is a complete connected Riemannian manifold such that
\begin{itemize}
 \item~$\mathcal M$ has a positive injectivity radius~$r_{\mathcal M}>0$,
\item~$\mathcal M$ has a sectional curvature bounded from above by a constant~$\mathcal K_{\max}$.
\end{itemize}
\end{ass}

For all~$t\geqslant 0$, a solution of~\eqref{eq-kinetic-model}~$\rho(t,\cdot)$ is then a probability measure on~$\mathcal M$. We still denote by~$K(\cdot,x_*,x_*')$ the probability law of two particles' position after they collide, if their initial positions were~$x_*,\,x_*'\in\mathcal M$. In this section, we only consider midpoint models, i.e. models that satisfy
\begin{ass}[midpoint models]\label{Ass-K1}
 For any~$x_*$ and~$x_*'$ in~$\mathcal M$, the support of the probability measure~$K(\cdot,x_*,x_*')$ is included in the set~$M_{x_*,x_*'}$ of midpoints of minimal geodesics joining~$x_*$ and~$x_*'$.
\end{ass}


The local stability of Dirac masses obtained in the case of spheres (see Theorem~\ref{theo-stability-dirac-sphere}) can be extended to this more general framework:
\begin{theo}
\label{theo-stability-dirac-midpoint-M}
Let the manifold~$\mathcal M$ satisfy Assumption~\ref{Ass-M}, and the collision kernel~$K$ satisfy Assumption~\ref{Ass-K1}. 
There exists~$C_1>0$ and~$\eta>0$ such that if~$\rho$ in~$C(\mathbb{R}_+,\mathcal P_2(\mathcal M))$ is a solution of~\eqref{eq-kinetic-model} with initial condition~$\rho_0$ that satisfies~$W_2(\rho_0,\delta_{x_0})< \eta$ for some~$x_0\in\mathcal M$, then there exists~$x_\infty\in \mathcal M$ such that
\[W_2(\rho_t,\delta_{x_\infty})\leqslant C_1  W_2(\rho_0,\delta_{x_0}) \, e^{-t/4}.\]
\end{theo}
Note that the convergence rate in Theorem~\ref{theo-stability-dirac-midpoint-M} is optimal, see Remark~\ref{rem:convergence_rate}.

The proof of this theorem is very close to the proof of Theorem~\ref{theo-stability-dirac-sphere}, since most of the proof of Theorem~\ref{theo-stability-dirac-sphere} was independent of the geometry of~$\mathbb{S}$. We will thus only emphasize the arguments that require some modifications. The main difficulty comes from the following lemma, which is a generalization of Lemma~\ref{lem:E-particles-local}:
\begin{lem}\label{lem:E-particles-local-M}
Let the manifold~$\mathcal M$ satisfy Assumption~\ref{Ass-M}, and the collision kernel~$K$ satisfy Assumption~\ref{Ass-K1}. 
There exists~$\kappa_1>0$ and a constant~$C_1\geqslant0$ such that for any~$\kappa\leqslant\kappa_1$, for any~$x_*,\,x_*',\,y\in \mathcal M$ such that
$\max\left(d(x_*,y),\,d(x_*',y),\,d(x_*,x_*')\right)\leqslant \kappa$, we have
\[\alpha(x_*,x_*',y) \leqslant -\frac14d(x_*,x_*')^2+ C_1\, \kappa^2\,d(x_*,x_*')^2.\]
\end{lem}
To prove this result, we should establish an estimate similar to~\eqref{appo}, in the context of a position set~$\mathcal M$ that is a Riemannian manifold. We prove such a result in the lemma below (Lemma~\ref{lem-apollonius-estimate-M-positive}). Just as in the case of Lemma~\ref{lem:E-particles-local}, the proof of Lemma~\ref{lem:E-particles-local-M} is trivial once the estimate~\eqref{estimate-apollonius-positive} is established, since~$\alpha(x_*,x_*',y)=m^2-\frac{1}2(b^2+b'^2)$ (where we have used the notations of the proof of Lemma~\ref{lem:E-particles-local}).

\begin{lem}\label{lem-apollonius-estimate-M-positive}
Let the manifold~$\mathcal M$ satisfy Assumption~\ref{Ass-M}. 
There exists~$\kappa_1>0$ and a constant~$C_1>0$ such that for any~$\kappa\leqslant\kappa_1$, if a geodesic triangle with side lengths~$2a,b,b'$ satisfies~$\max(2a,b,b')\leqslant\kappa$, then, denoting by~$m$ the length of the median corresponding to the side of length~$2a$ (as in figure~\ref{fig-apollonius-sphere}), we have
\begin{equation}\label{estimate-apollonius-positive}
m^2+ a^2 -\frac{b^2+b'^2}2 \leqslant C_1\, \mathcal K_{\max} \, \kappa^2 a^2.
\end{equation}
\end{lem}

To prove Lemma~\ref{lem-apollonius-estimate-M-positive}, we will use the Rauch comparison theorem, that we recall here:
\begin{theo}[\cite{berger2003panoramic}*{last part of Theorem~73},~\cite{cheeger1975comparison}*{Corollary~$1.30$}]\label{lem-rauch}
Let the manifold~$\mathcal M$ satisfy Assumption~\ref{Ass-M}, and~$r=\min(r_{\mathcal M},\frac{\pi}{\sqrt{K_{\max}}})$. Let~$(x,y,z)$ be a minimal geodesic triangle on~$\mathcal M$ such that the geodesic joining~$x$ and~$y$ is included in the geodesic ball of radius~$r$ centered at~$z$.
Let~$(\bar x,\bar y,\bar z)$ be a minimal geodesic triangle on the sphere of constant curvature~$\mathcal K_{\max}$ (or the plane if~$\mathcal K_{\max}=0$) such that~$d(x,z)=d(\bar x,\bar z)$ and~$d(y,z)=d(\bar y,\bar z)$, the angle between the geodesics at~$\bar z$ being the same as the one at~$z$ in the geodesic triangle of~$\mathcal M$. Then~$d(x,y)\geqslant d(\bar x,\bar y)$.
\end{theo}

\begin{proof}[Proof of Lemma~\ref{lem-apollonius-estimate-M-positive}] 
We first note that the case of the unit sphere ($\mathcal M=\mathbb{S}$,~$\mathcal K_{\max}=1$) corresponds exactly to~\eqref{appo}, in the proof of Lemma~\ref{lem:E-particles-local}, which is valid for any~$\kappa_1<\frac{2\pi}3$. A simple scaling argument (dividing all distances by~$\sqrt{\mathcal K_{\max}}$) shows that this is also the case for the sphere of constant curvature~$\mathcal K_{\max}>0$ (for any~$\kappa_1<\frac{2\pi}{3\sqrt{\mathcal K_{\max}}}$). Moreover, the case of the Euclidean space corresponds to the Apollonius theorem on the plane, for which the inequality is an equality.

%

\medskip

In the general case, we take~$\kappa_1<\frac23r$ where~$r=\min\left(r_{\mathcal M},\frac{\pi}{\sqrt{\mathcal K_{\max}}}\right)$ and~$C_1$ given by the case of the sphere of constant curvature~$\mathcal K_{\max}$ (or the Euclidean space if~$\mathcal K_{\max}=0$).

Let~$(x,x',y)$ be a geodesic triangle in~$\mathcal M$ with side lengths~$2a$,~$b$, and~$b'$ (see Figure~\ref{fig-rauch-M-sphere}). We denote by~$z$ the midpoint of the geodesic joining~$x$ to~$x'$ and by~$m$ the distance~$d(y,z)$. Finally, we write~$\theta$ the angle at the point~$z$ in the triangle~$(x,y,z)$, therefore the angle at~$z$ in the triangle~$(x',y,z)$ is~$\pi-\theta$ (see figure~\ref{fig-rauch-M-sphere}). We suppose that~$\max(2a,b,b')=\kappa\leqslant\kappa_1$, therefore we get~$a+b\leqslant\frac32\kappa_1<r$ and by triangle inequality the geodesic arc from~$x$ to~$y$ is included in the ball of center~$z$ and radius~$r$.
\begin{figure}[h]
\centering
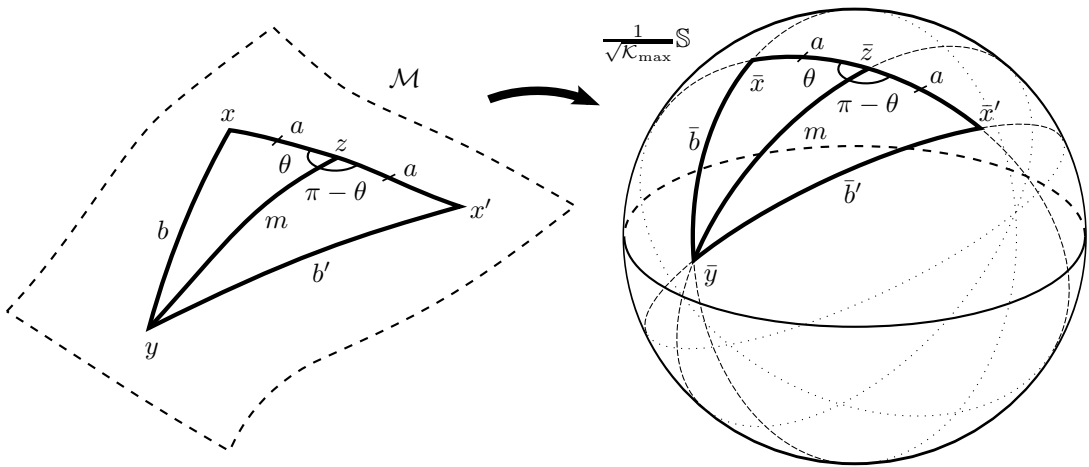
\caption{Triangle configurations on~$\mathcal M$ and on the sphere of constant curvature~$\mathcal K_{\max}$.}
\label{fig-rauch-M-sphere} 
\end{figure}
We now take~$(\bar x,\bar x',\bar y)$ a triangle on the sphere of constant curvature~$\mathcal K_{\max}$ with the same configuration for the lengths~$a$,~$m$ and the angles~$\theta$ and~$\pi-\theta$, and we denote by~$\bar b$ and~$\bar b'$ the distances~$d(\bar x,\bar y)$ and~$d(\bar x',\bar y)$ (see Figure~\ref{fig-rauch-M-sphere}). We can apply Theorem~\ref{lem-rauch} and we get~$b\geqslant\bar b$ and~$b'\geqslant\bar b'$. Therefore if~$\max(2a,b,b')\leqslant\kappa\leqslant\kappa_1$, we get~$\max(2a,\bar b,\bar b')$ and we apply the estimate~\eqref{estimate-apollonius-positive} in the case of the sphere of constant curvature~$\mathcal K_{\max}$ (or the Euclidean space if~$\mathcal K_{\max}=0$) to get
\[m^2+a^2-\frac{b^2+b'^2}2\leqslant m^2+a^2-\frac{\bar b^2+\bar b'^2}2\leqslant C_1\, \mathcal K_{\max}\, \kappa^2a^2,\]
which ends the proof of Lemma~\ref{lem-apollonius-estimate-M-positive}.
\end{proof}

We can now describe the proof of the main result of the section:
\begin{proof}[Proof of Theorem~\ref{theo-stability-dirac-midpoint-M}]
We first remark that the good framework is~$\mathcal P_2(\mathcal M)$, the set of probability measures~$\rho$ such that there exists~$x\in\mathcal M$ with~$W_2(\rho,\delta_x)<+\infty$. We recall that the Wasserstein distance~$W_2$ is a distance on~$\mathcal P_2(\mathcal M)$. We can define the energy~$E(\rho)$ in the same way as~\eqref{def-energy}:
\begin{equation}
  E(\rho)=\int_{\mathcal M\times\mathcal M} d(x,y)^2\,d\rho(x)\,d\rho(y),
\label{def-energy-M}
\end{equation}
and we get that Lemma~\ref{lem-cheb-sphere} is still valid, replacing~$\mathbb{S}$ by~$\mathcal M$ and~$\mathcal P(\mathbb{S})$ by~$\mathcal P_2(M)$. In particular, the estimates~\eqref{E-4W2} and~\eqref{Wxbar} give that for any probability measure~$\rho$, we have that~$\rho\in\mathcal P_2(\mathcal M)$ if and only if~$E(\rho)$ is finite.

The time derivative is still given by~\eqref{dEdt}, where~$\alpha$ is given by~\eqref{def-alpha} and~$\mathbb{S}$ is replaced by~$\mathcal M$ in these two expressions. The global estimate given by Lemma~\ref{lem:E-particles-global} holds, and the proof, only relying on triangle inequalities, is unchanged. The counterpart of local estimate provided by Lemma~\ref{lem:E-particles-local} is now exactly given by Lemma~\ref{lem-apollonius-estimate-M-positive}. Therefore we obtain the same proposition as Proposition~\ref{prop-decay-energy}, replacing~$\mathcal P(\mathbb{S})$ by~$\mathcal P_2(M)$.

The control of the displacement of~$\bar x$ is performed exactly in the same way as in subsection~\ref{subsec-local-stability}. The only detail to check is that~$M_{t,t'}= \sup_{s\in[t,t']}d(\bar x(t),\bar x(s))$ is finite (we cannot use here the fact that~$\mathcal M$ is compact). And indeed we have
\begin{align*}
d(\bar x(t),\bar x(s))&=W_2(\delta_{\bar x(t)},\delta_{\bar x(s)})\\
&\leqslant W_2(\delta_{\bar x(s)},\rho_s)+W_2(\rho_s,\rho_t)+W_2(\rho_t,\delta_{\bar x(t)}) \\
&\leqslant\sqrt{E(\rho_s)} + \sqrt{E(\rho_t)} + W_2(\rho_s,\rho_t)\\
&\leqslant2\sqrt{E(\rho_0)} + W_2(\rho_s,\rho_t),
\end{align*}
thanks to Lemma~\ref{lem-cheb-sphere} (or rather the generalization of that result to~$\mathcal M$). This estimate implies that~$s\mapsto d(\bar x(t),\bar x(s))$ is bounded on the compact~$[t,t']$, since~$\rho\in C(\mathbb{R}_+,\mathcal P_2(\mathcal M))$.
\end{proof}

\section{Extension to non midpoint models}\label{section:non-midpoint}

\subsection{Contracting collision kernels}\label{subsection:contracting-kernels}

Theorem~\ref{theo-stability-dirac-midpoint-M} shows that the local stability of Dirac masses for the midpoint model on the sphere (Theorem~\ref{theo-stability-dirac-sphere}) can be generalized to a Riemannian manifold~$\mathcal M$. One can wonder whether this stability result can also be obtained under weaker assumptions on the collision kernel~$K$. 
%
%
%
For instance, a possible generalization of midpoint models (see Assumption~\ref{Ass-K1}) is the case of a ``noisy'' interaction kernel, of the form
\begin{equation}
\widetilde K_\chi(x,x_*,x_*')=\frac12\int_{\mathcal M}[\chi(x_*,x_*',y)K(x,y,x_*')+\chi(x_*',x_*,y)K(x,x_*,y)]\,dy,\label{K-noisy}
\end{equation}
where~$K$ is a midpoint kernel, and~$y\mapsto\chi(x_*,x_*',y)$ is a probability measure for any~$x_*,\,x_*'\in\mathcal M$. From the modeling point of view,~$y\mapsto \chi(x_*,x_*',y)$ represents the error that an individual located at~$x_*$ makes when it estimates the position of an individual actually located at~$x_*'$. For instance, if we consider that the error a particle (that could be e.g. a bird) makes when it collides with another particle follows a uniform law in a geodesic ball of radius~$\varepsilon$ centered at the actual position of the other individual, the corresponding term~$\chi$ would be:
\begin{equation}
\chi_\varepsilon(x_*,x_*',y)=\frac{\mathbbm{1}_{d(x_*',y)\leqslant\varepsilon}}{\int_{\{y|d(x_*',y)\leqslant \varepsilon\}}dy}.\label{chi-epsilon}
\end{equation}
Another possibility is to assume that the estimation of the other particle's position is precise, but that there is an error in the selection of the post-collision position. This latter assumption was made in 
the original model of Bertin, Droz and Grégoire on the circle~\cite{bertin2006boltzmann}.
The corresponding kernel would then be
\begin{equation}\label{K-noisy2}
 \widetilde K_{BDG}(x,x_*,x_*')=\int_{\mathcal M}\chi_{BDG}(x,y,x_*)K(y,x_*,x_*')\,dy,
\end{equation}
 where~$\chi_{BDG}(x,y,x_*)$ is now the probability of obtaining an post-collision position~$y$, knowing that the midpoint is precisely~$x$ (this probability may depend on the pre-collision position~$x_*$). 

\begin{rem}
 One can check that the collision kernel~$K$ defined by~\eqref{K-noisy}-\eqref{chi-epsilon} satisfies the condition of Proposition~\ref{prop-existence-uniqueness}. 

On the contrary, it is not the case of the second model~\eqref{K-noisy2}, even when~$\chi_{BDG}$ is smooth: consider~$\mathcal M=\mathbb{S} ^1$, the map~$x_*'\mapsto\widetilde K_{BDG}(\cdot,x_*,x_*')$, from~$\mathcal M$ into~$\left(\mathcal P_2(\mathcal M),W_2\right)$ is then discontinuous for~$x_*'$ in a neighborhood of~$-x_*'$, just like the midpoint model.
\end{rem}

%

A natural extension of Theorem~\ref{theo-stability-dirac} would then be to show that for~$\varepsilon>0$ small, the solution of the kinetic equation with a kernel defined by~\eqref{K-noisy}-\eqref{chi-epsilon} converges exponentially fast to a given profile. This problem seems to be difficult, since the convexity arguments used in related problems (see e.g.~\cites{carrillo2006contractions, bolley2013uniform}) do not apply directly when the geometry of the set of positions is complex (for instance when~$\mathcal M$ is a sphere). We are not able to prove such a general result here. 
However, for a class of kernels, called ``contracting kernels'', for which the Dirac masses are stationary states, we are able to extend the result of Theorem~\ref{theo-stability-dirac-midpoint-M}. 
We define as ``contracting kernels'' the collision kernels~$K$ satisfying the following property:
\begin{ass}[Contraction property]\label{Ass:contraction}
 There exists~$\beta\in[0,1)$ such that for any~$x_*,x_*'\in\mathcal M$,
\[\int_{\mathcal M} d(x,x_*)^2 K(x,x_*,x_*') dx \leqslant \tfrac14(1+\beta)d(x_*,x_*')^2,\]
\end{ass}
Assumption~\ref{Ass:contraction} corresponds to the fact that, when a particle located at~$x_*'$ interact with another one located at~$x_*$, the mean squared distance to~$x_*$ is decreased by a factor~$\frac14(1+\beta)$. We propose below a slightly more restrictive condition, which appears as a generalization of the midpoint model studied in Section~\ref{section:manifold} (see Assumption~\ref{Ass-K1}):
\begin{ass}[Midpoint contraction property]\label{Ass:midpoint_prop}
There exists~$\widetilde{\beta}\in[0,1)$ such that for any~$x_*,x_*'$, we have
\[\int_{\mathcal M} d(x,M_{x_*,x_*'})^2 K(x,x_*,x_*') dx \leqslant \widetilde{\beta}\,\frac{d(x_*,x_*')^2}4 ,\]
where~$M_{x_*,x_*'}$ is the set of midpoints of minimal geodesics joining~$x_*$ and~$x_*'$, and~$d(x,M_{x_*,x_*'})$ is the geodesic distance between~$x$ and this set.
\end{ass}
Assumption~\ref{Ass:midpoint_prop} expresses that, after a collision of particles located at~$x_*$ and~$x_*'$, the mean squared distance to the ``nearest midpoint'' of~$x_*$ and~$x_*'$ is decreased by a factor~$\widetilde{\beta}$. At first glance, it is not clear that Assumption~\ref{Ass:contraction} ensures that the model is close to a midpoint model. This appears more clearly with Assumption~\ref{Ass:midpoint_prop}, which is why we have introduced this assumption here. 

From a modeling point of view, Assumption~\ref{Ass:contraction} is satisfied when two interacting particles are not necessarily able to compute precisely their midpoint, but are able to reduce their distance (by a fixed factor) after the collision. We will prove in the next proposition that Assumption~\ref{Ass:midpoint_prop} implies Assumption~\ref{Ass:contraction} (at least if~$\widetilde{\beta}+2\sqrt{\widetilde{\beta}}<1$), while it will appear in the next subsection that Assumption~\ref{Ass:contraction} is the precise condition that is needed to extend the results of Theorem~\ref{theo-stability-dirac-midpoint-M}. 


\begin{prop}
\label{prop-betatilde}~
\begin{itemize}
\item Assumption~\ref{Ass:midpoint_prop} with~$\widetilde{\beta}=0$ is equivalent to Assumption~\ref{Ass:contraction} with~$\beta=0$, and to the fact that~$K$ is a midpoint kernel, that is Assumption~\ref{Ass-K1}.
\item When~$\mathcal M=\mathbb{R}^n$, Assumption~\ref{Ass:midpoint_prop} is equivalent to Assumption~\ref{Ass:contraction} with~$\beta=\widetilde{\beta}$.
\item Assumption~\ref{Ass:midpoint_prop} implies Assumption~\ref{Ass:contraction} with~$\beta=\widetilde{\beta}+2\sqrt{\widetilde{\beta}}$.
\end{itemize}
\end{prop}

This class includes some of the kernels of the form~\eqref{K-noisy} and~\eqref{K-noisy2}.

\begin{proof}[Proof of Proposition~\ref{prop-betatilde}]
For the first point, Assumption~\ref{Ass:midpoint_prop} with~$\widetilde{\beta}=0$ is obviously equivalent to the fact that the support of~$K(\cdot,x_*,x_*')$ is included in~$M_{x_*,x_*'}$. If the support of~$K(\cdot,x_*,x_*')$ is included in~$M_{x_*,x_*'}$, then since~$d(x,x_*)=\frac12 d(x_*,x_*')$, for all~$x\in M_{x_*,x_*'}$, we immediately get the estimate of Assumption~\ref{Ass:contraction} with~$\beta=0$.

Conversely, we have, for~$x\in\mathcal M$:
\begin{equation}
\label{eq-ti-cs}\tfrac14d(x_*,x_*')^2\leqslant\tfrac14[d(x_*,x)+d(x,x_*')]^2\leqslant\tfrac12[d(x_*,x)^2+d(x,x_*')^2],
\end{equation}
and these inequalities are equalities only if~$d(x_*,x_*')=d(x_*,x)+d(x,x_*')$ and~$d(x_*,x)=d(x_*',x)$, that is to say~$x$ is on a minimal geodesic between~$x_*$ and~$x_*'$ and it is precisely on the midpoint of this geodesic, i.e.~$x\in M_{x_*,x_*'}$. 
So if Assumption~\ref{Ass:contraction} is satisfied with~$\beta=0$, since the measures~$K(\cdot,x_*,x_*')$ and~$K(\cdot,x_*',x_*)$ are equal, integrating~\eqref{eq-ti-cs} against~$K(\cdot,x_*,x_*')$, we get~$\tfrac14d(x_*,x_*')^2\leqslant\tfrac14d(x_*,x_*')^2$. 
Therefore for all~$x$ in the support of~$K(\cdot,x_*,x_*')$, we have~$x\in M_{x_*,x_*'}$.

The second point comes from Apollonius theorem. Indeed in~$\mathbb{R}^n$,~$M_{x_*,x_*'}$ is reduced to the single point~$x_m=\frac12(x_*+x_*')$, and we have
\[|x-x_m|^2=\frac12(|x-x_*|^2+|x-x_*'|^2)+\frac14|x_*-x_*'|^2.\]
Integrating with respect to~$K(\cdot,x_*,x_*')$ gives the equivalence between Assumption~\ref{Ass:contraction} and Assumption~\ref{Ass:midpoint_prop} when$\mathcal M=\mathbb{R}^n$ and~$\widetilde{\beta}=\beta$.

The third point comes from the Cauchy-Schwarz inequality and the fact that~$K(\cdot,x_*,x_*')$ is a probability measure. 
Indeed, for~$x\in\mathcal M$ and~$y\in M_{x_*,x_*'}$, we have
\[d(x,x_*)\leqslant d(x,y)+d(y,x_*)=d(x,y)+\frac12d(x_*,x_*').\]
Therefore we have~$d(x,x_*)\leqslant d(x,M_{x_*,x_*'})+\frac12d(x_*,x_*')$, and
\[d(x,x_*)^2\leqslant d(x,M_{x_*,x_*'})^2+d(x,M_{x_*,x_*'})d(x_*,x_*')+\frac14d(x_*,x_*')^2.\]
We now integrate this inequality with respect to~$K(\cdot,x_*,x_*')$, and use the Cauchy-Schwarz inequality to estimate the second term. Since~$K(\cdot,x_*,x_*')$ is a probability measure, when Assumption~\ref{Ass:midpoint_prop} holds true, we get
\[\int_{\mathcal M} d(x,x_*)^2\,K(x,x_*,x_*')\,dx\leqslant\frac{\beta d(x_*,x_*')^2}4+\sqrt{\frac{\beta d(x_*,x_*')^2}4}d(x_*,x_*')+\frac14d(x_*,x_*')^2,\]
which gives Assumption~\ref{Ass:contraction} with~$\beta=\widetilde{\beta}+2\sqrt{\widetilde{\beta}}$.
\end{proof}

In order to enounce our main result, we will also need the two following technical assumptions on the collision kernel~$K$:

\begin{ass}[Midpoint symmetry for small distances]\label{Ass:H2}
There exists~$\kappa_0\in (0, r_{\mathcal M}]$ such that for any~$x_*,x_*'\in\mathcal M$ with~$d(x_*,x_*')<\kappa_0$, if we denote by~$x_m$ the midpoint of the (unique) geodesic arc between~$x_*$ and~$x_*'$, the probability~$K(\cdot,x_*,x_*')$ is symmetric with respect to~$x_m$ in the (open) geodesic ball~$B$ of center~$x_m$ and radius~$\kappa_0$. 
More precisely, we define~$S_{x_m}$ as the function mapping a point~$x\in B$ to its symmetric~$x'\in B$ (at the same distance from~$x_m$ and on the same geodesic passing through~$x_m$ as~$x$ to~$x_m$), that is to say~$S_{x_m}=\exp_{x_m}\circ(-\mathrm{Id})\circ exp_{x_m}^{-1}$, where~$\exp_{x_m}$ is the exponential map from the tangent space~$T_{x_m}\mathcal M$ to~$\mathcal M$.
The assumption on~$K$ consists in saying that for all bounded continuous functions~$f:B\to\mathbb{R}$, we have:
\[\int_{B} f(x) K(x,x_*,x_*') dx=\int_{B} f(S_{x_m}(x)) K(x,x_*,x_*') dx.\]
\end{ass}
Notice that in the above assumption, we have used the fact that the minimal geodesic arc connecting two points on a manifold is unique, provided the two points are close.
\begin{ass}[Higher moment control]\label{Ass:H3}
 There exists~$p>2$ and~$C>0$ such that for any~$x_*,x_*'$,
\[\int_{\mathcal M} d(x,x_*)^p K(x,x_*,x_*') dx \leqslant C\, d(x_*,x_*')^p.\]
\end{ass}

%
Assumption~\ref{Ass:H3} is not very restrictive, and prevents any bad behavior of~$K$ far from the minimal geodesics connecting~$x_*$ to~$x_*$.  Assumption~\ref{Ass:H2} is more restrictive. We will need it in some of the estimates in the proof, although it is not a necessary condition to obtain the local stability of Dirac masses. For instance, for the simple non symmetric model on~$\mathbb{R}$ given by~$K(\cdot,x_*,x_*')=\delta_{\frac12(x_*+x_*')+\gamma|x_*-x_*'|}$, it is possible (through the estimation of the first and second moment) to show that the solution converges to a Dirac mass with exponential rate when~$\gamma$ is sufficiently small, even though~$K$ does not satisfy Assumption~\ref{Ass:H2}. We will however not try to relax the Assumption~\ref{Ass:H2} in the present work.

\medskip

Finally, let us present some simple collision kernels that satisfy the properties stated in this subsection
\begin{exa}[examples of contracting kernels]~
 \begin{itemize}
  \item The midpoint model (see Assumption~\ref{Ass-K1}) satisfies Assumption~\ref{Ass:contraction} with~$\beta=0$, Assumption~\ref{Ass:H2} with~$\kappa_0=r_{\mathcal M}$, and Assumption~\ref{Ass:H3} with~$C=2^{-p}$, for any~$p>2$.

\item If we define a ``noisy interaction'' kernel as in~\eqref{K-noisy}, with
\begin{equation}
\chi(x_*,x_*',y)=\frac{\mathbbm{1}_{d(x_*',y)\leqslant\gamma d(x_*,x_*')}}{\int_{\{y|d(x_*',y)\leqslant \gamma d(x_*,x_*')\}}dy},\label{chi-beta}
\end{equation}
then it can be shown, using only simple triangle inequalities, that it satisfies Assumption~\ref{Ass:H3} for any~$p$ with~$C=\frac{(1+3\gamma)^p+(1+\gamma)^p}{2^{p+1}}$, and therefore it also satisfies Assumption~\ref{Ass:contraction} with~$\beta=9\gamma^2+8\gamma$ ($\beta<1$ is then equivalent to~$\gamma<\frac19$). In the case where the manifold~$\mathcal M$ is the sphere, Assumption~\ref{Ass:H2} is also satisfied, thanks to the symmetries of the model. 
\item If we define a ``noisy interaction'' kernel as in~\eqref{K-noisy2}, with
\begin{equation*}
\chi_{BDG}(x,y,x_*)=\frac{\mathbbm{1}_{d(x,y)\leqslant\gamma d(x,x_*)}}{\int_{\{y|d(x,y)\leqslant \gamma d(x,x_*)\}}dy},
\end{equation*}
it is also possible to show that it satisfies Assumption~\ref{Ass:midpoint_prop}, and therefore Assumption~\ref{Ass:contraction}.
 \end{itemize}

\end{exa}


\subsection{Local stability of Dirac masses on manifolds, for contracting collision kernels}

The main result concerns the nonlinear stability of Dirac masses, provided the collision kernel~$K$ satisfies the conditions described in the previous subsection. To study this class of models, we will also need to slightly strengthen the assumption we made on the manifold~$\mathcal M$: 
\begin{ass}[Position set~$\mathcal M$]\label{Ass-M2}
$\mathcal M$ is a complete connected Riemannian manifold such that
\begin{itemize}
 \item~$\mathcal M$ has a positive injectivity radius~$r_{\mathcal M}>0$,
\item~$\mathcal M$ has a sectional curvature bounded from above and from below by positive constants.
\end{itemize}
\end{ass}

The main result of this section is then the following:
\begin{theo}
\label{theo-stability-dirac}
Let the position set~$\mathcal M$ satisfy Assumption~\ref{Ass-M2}, and the collision kernel~$K$ satisfy Assumptions~\ref{Ass:contraction},~\ref{Ass:H2} and~\ref{Ass:H3}. There exists~$C_1>0$ and~$\eta>0$ such that for any solution~$\rho$ of~\eqref{eq-kinetic-model} in~$C(\mathbb{R}_+,\mathcal P_2(\mathcal M))$ with initial condition~$\rho_0$ satisfying~$W_2(\rho_0,\delta_{x_0})< \eta$ for some~$x_0\in\mathcal M$, there exists~$x_\infty\in \mathcal M$ such that
\[W_2(\rho_t,\delta_{x_\infty})\leqslant C_1  W_2(\rho_0,\delta_{x_0}) \, e^{-\frac14(1-\beta)t}.\]
\end{theo}

Notice that the convergence rate is here lower than in the previous midpoint model. This is the price to pay for the lower contractivity of the collision model.

  Just as in Section~\ref{section:manifold}, the proof of Theorem~\ref{theo-stability-dirac} is very similar to the proof of Theorem~\ref{theo-stability-dirac-midpoint-M}, and we will only emphasize the changes. The only thing to do is to obtain the same kind of estimates for the quantity~$\alpha$ (given by~\eqref{def-alpha}). The global estimate given by Lemma~\ref{lem:E-particles-global} in the case of the midpoint model on the sphere is simply based on triangle inequalities. Those arguments combined to Assumption~\ref{Ass:contraction} allow us to obtain a similar global estimate here, although with different constants:
\begin{lem}\label{lem-global-beta}
Let~$K$ satisfy Assumption~\ref{Ass:contraction}. For any~$x_*,\,x_*',\,y\in \mathcal M$, we have 
\begin{equation}
\alpha(x_*,x_*',y)\leqslant-\frac{1-\beta}4\,d(x_*,x_*')^2 + (1+\sqrt{1+\beta})\,d(x_*,x_*') \min\big(d(x_*,y),\,d(x_*',y)\big).\label{estimate-d1-beta}
\end{equation}
\end{lem}
\begin{proof}[Proof of Lemma~\ref{lem-global-beta}]
The triangle inequality in the triangle~$(x_*,x,y)$ gives
\[d(x,y)^2\leqslant d(x_*,y)^2+2\,d(x_*,y)\,d(x,x_*)+d(x,x_*)^2,\]
which, after integrating against the probability measure~$K(\cdot,x_*,x_*')$ and using the Cauchy-Schwarz inequality and Assumption~\ref{Ass:contraction}, gives
\[\int_{\mathcal M} d(x,y)^2 K(x,x_*,x_*')\,dx\leqslant d(x_*,y)^2+\sqrt{1+\beta}\,d(x_*,y)\,d(x_*,x_*')+\tfrac14(1+\beta)\,d(x_*,x_*')^2.\]
We now write~$d(x_*',y)^2\geqslant d(x_*,y)^2-2\,d(x_*,y)\,d(x_*',x_*)+d(x_*',x_*)^2$ and insert these two estimates in the expression~\eqref{def-alpha} of~$\alpha$. 
We obtain
\[\alpha(x_*,x_*',y)\leqslant-\frac{1-\beta}4\,d(x_*,x_*')^2 +(1+\sqrt{1+\beta})\,d(x_*,y)\,d(x_*',x_*).\]
Since~$x_*$ and~$x_*'$ play symmetric roles, this ends the proof of the Lemma.
\end{proof}


We can also establish a local estimate, similar to Lemma~\ref{lem:E-particles-local}.

\begin{lem}\label{lem-local-beta}
Let the position set~$\mathcal M$ satisfy Assumption~\ref{Ass-M2}, and the kernel~$K$ satisfy Assumptions~\ref{Ass:contraction},~\ref{Ass:H2} and~\ref{Ass:H3}.
There exists~$\kappa_2\leqslant\kappa_0$ and~$C_p>0$ such that for any~$x_*,\,x_*',\,y\in \mathcal M$ such that~$\max\left(d(x_*,y),\,d(x_*',y),\,d(x_*,x_*')\right)\leqslant \kappa$, we have
\begin{equation}
\alpha(x_*,x_*',y) \leqslant -\frac{1-\beta}4\,d(x_*,x_*')^2 + C_p\,\kappa^{2(1-\frac2p)}\,d(x_*,x_*')^2,\label{estimate-d2-beta}
\end{equation}
where~$p$ and~$\kappa_0$ are given by Assumption~\ref{Ass:H2} and Assumption~\ref{Ass:H3}.
\end{lem}

To prove this result, we need the following technical estimate:
\begin{lem}\label{lem-apollonius-estimate-M-negative}
Assume that~$\mathcal M$ satisfies Assumption~\ref{Ass-M2}.
We fix~$\kappa_0>0$. There exists~$C_1\geqslant0$ such that for any~$\kappa\leqslant\kappa_0$, if a minimal geodesic triangle with side lengths~$2a,b,b'$ satisfies~$\max(2a,b,b')\leqslant\kappa$, then, denoting by~$m$ the length of the median corresponding to the side of length~$2a$ (as in figure~\ref{fig-apollonius-sphere}), we have
\begin{equation}\label{estimate-apollonius-negative}
m^2+ a^2 -\frac{b^2+b'^2}2 \geqslant C_1\, \mathcal K_{\min} \, \kappa^2 a^2.
\end{equation}
\end{lem}
\begin{rem}
 This inequality is indeed related to estimate~\eqref{estimate-apollonius-positive}: it is an estimate of the quantity~$m^2+ a^2 -\frac{b^2+b'^2}2$ from below, while~\eqref{estimate-apollonius-positive} is an estimate of the same quantity from above.

Notice that to obtain Lemma~\ref{lem-apollonius-estimate-M-negative}, we only need that the sectional curvature of~$\mathcal M$ is bounded from below by a constant~$\mathcal K_{\min}\leqslant0$.
\end{rem}

To prove Lemma~\ref{lem-apollonius-estimate-M-negative}, we will use the Toponogov theorem (see below), which is similar to the Rauch Theorem (Theorem ~\ref{lem-rauch}), with a reverse inequality.

\begin{theo}[Toponogov Theorem]\label{lem-toponogov}\cite{berger2003panoramic}*{Theorem~73} or~\cite{cheeger1975comparison}*{Theorem~2.2~(B)}.
We suppose that the sectional curvature is bounded below by~$\mathcal K_{\min}\leqslant0$.
We suppose~$(x,y,z)$ is a minimal geodesic triangle on~$\mathcal M$.
Let~$(\bar x,\bar y,\bar z)$ be a minimal geodesic triangle on the hyperbolic plane of constant curvature~$\mathcal K_{\min}$ (or the plane if~$\mathcal K_{\min}=0$) such that~$d(x,y)=d(\bar x,\bar y)$ and~$d(x,z)=d(\bar x,\bar z)$, the angle at~$\bar x$ being the same as the one at~$x$.

Then~$d(y,z)\leqslant d(\bar y,\bar z)$.
\end{theo}

\begin{proof}[Proof of Lemma~\ref{lem-apollonius-estimate-M-negative}]

We then start with the case of the hyperbolic plane with constant curvature~$-1$, and as before we will obtain the general case by scaling and comparison.

There is an equivalent of Apollonius formula in hyperbolic geometry (similarly to~\eqref{apollonius-spherical}, which holds in the spherical geometry). Indeed, the hyperbolic trigonometry formula (see for instance~\cite{berger2003panoramic}*{eq.~$(4.21)$}), in the configuration of Figure~\ref{fig-rauch-M-sphere} gives~$\cosh b=\cosh a \cosh m - \cos \theta\sinh b \sinh m$, and then we get
\begin{equation}
\frac{\cosh b + \cosh b'}2=\cosh a \cosh m.\label{apollonius-hyperbolic}
\end{equation} 
The same argument as in the proof of Lemma~\ref{lem-apollonius-estimate-M-positive}, with the function~$h:t\mapsto\mathrm{arcosh}(t)^2$ (which is now concave) instead of~$g$ gives that
\begin{align*}
\frac{b^2+b'^2}2 &\leqslant m^2 + (\cosh a-1)\,\frac{2\,m \cosh m}{\sinh m}\\
&\leqslant m^2 + a^2(1+O(a^2))(1+O(m^2)),
\end{align*}
which gives the estimate~\eqref{estimate-apollonius-negative} since~$m\leqslant a+b$. The same scaling argument provides the estimate for the hyperbolic plane of constant curvature~$\mathcal K_{\min}<0$. Again, the case of~$\mathcal K_{\min}=0$ corresponds to Apollonius theorem on the plane. Finally, the same comparison arguments works when we use the Toponogov comparison theorem, i.e. Theorem~\ref{lem-toponogov} (applied to the same exact configuration as in Figure~\ref{fig-rauch-M-sphere}, except that the sphere is replaced by the hyperbolic plane), and this ends the proof of Lemma~\ref{lem-apollonius-estimate-M-negative}.
\end{proof}

We can now prove the main local estimate of the quantity~$\alpha$.
\begin{proof}[Proof of Lemma~\ref{lem-local-beta}]
The proof relies mainly in two applications of Lemma~\ref{lem-apollonius-estimate-M-positive}, one application of Lemma~\ref{lem-apollonius-estimate-M-negative}, and some Markov inequalities using the control provided by Assumption~\ref{Ass:H3} on higher moments.

We fix~$0<\varepsilon<1$ and~$\kappa_2\leqslant\kappa_0$ such that~$(\kappa_2)^{\varepsilon}\leqslant\kappa_0$. We now take~$\kappa\leqslant\kappa_2$ and~$x_*$,~$x_*'$, and~$y$ in~$\mathcal M$ such that~$\max\left(d(x_*,y),\,d(x_*',y),\,d(x_*,x_*')\right)\leqslant \kappa$. All along this proof, we will denote by~$C$ a generic constant which does not depend on~$\kappa$. We then define, as Assumption~\ref{Ass:H2}, the midpoint~$x_m$ of the (unique) minimal geodesic arc joining~$x_*$ and~$x_*'$, and the geodesic ball~$B_\varepsilon$ of center~$x_m$ and radius~$\kappa^\varepsilon$ (which is less than~$\kappa_0$). We recall the expression~\eqref{def-alpha} :
\[\alpha(x_*,x_*',y)=\int_{\mathcal M}\left[d(x,y)^2-\frac{d(x_*,y)^2+d(x_*',y)^2}2\right] K(x,x_*,x_*') \, dx.\] 

We will split the integral depending whether~$x$ is in~$B_\varepsilon$ or not. 
We first estimate the integral outside~$B_\varepsilon$, using triangle inequalities and Markov inequality. Indeed, if~$d(x_*,y)\leqslant\kappa$,~$d(x_*,x_m)\leqslant\frac12\kappa$ and~$d(x,x_m)\geqslant\kappa^\varepsilon$, provided we choose~$\kappa_2$ sufficiently small, we have~$d(x,y)\leqslant Cd(x,x_*)$ and~$d(x,x_*)\geqslant C \kappa^\varepsilon$. So if we write
\begin{equation}
\alpha_\varepsilon(x_*,x_*',y)=\int_{B_\varepsilon}\left[d(x,y)^2-\frac{d(x_*,y)^2+d(x_*',y)^2}2\right] K(x,x_*,x_*') \, dx,\label{def-alpha-epsilon}
\end{equation}
we get
\begin{align*}
\alpha(x_*,x_*',y)-\alpha_\varepsilon(x_*,x_*',y)&=\int_{B_\varepsilon^c}\left[d(x,y)^2-\frac{d(x_*,y)^2+d(x_*',y)^2}2\right] K(x,x_*,x_*') \, dx\\
&\leqslant C\int_{\mathcal M\setminus B_\varepsilon} d(x,x_*)^2 K(x,x_*,x_*') \, dx\\
&\leqslant C (\kappa^{-\varepsilon})^{p-2}\int_{\mathcal M\setminus B_\varepsilon} d(x,x_*)^p K(x,x_*,x_*') \, dx\\
&\leqslant C \kappa^{(2-p)\varepsilon}\,d(x_*,x_*')^p,
\end{align*}
thanks to Assumption~\ref{Ass:H3}. Since~$d(x_*,x_*')\leqslant\kappa$, we obtain
\begin{equation}
\alpha(x_*,x_*',y)-\alpha_\varepsilon(x_*,x_*',y)\leqslant C\kappa^{(p-2)(1-\varepsilon)}\,d(x_*,x_*')^2.\label{est-outside-B}
\end{equation}

We now estimate~$\alpha_\varepsilon(x_*,x_*',y)$ given by~\eqref{def-alpha-epsilon}. Thanks to Assumption~\ref{Ass:H2}, we get
\begin{equation}
\alpha_\varepsilon(x_*,x_*',y)=\int_{B_\varepsilon}\left[\frac{d(x,y)^2+d(S_{x_m}(x),y)^2}2-\frac{d(x_*,y)^2+d(x_*',y)^2}2\right] K(x,x_*,x_*') \, dx.\label{alpha-epsilon}
\end{equation}
For~$x$ in~$B_\varepsilon$, we write~$x'=S_{x_m}(x)$. We denote by~$b$ (resp.~$b'$,~$b_*$,~$b_*'$) the distance between~$y$ and~$x$ (resp.~$x'$,~$x_*$,~$x_*'$), and we write~$a=d(x_m,x)=d(x_m,x')$,~$a*=d(x_m,x_*)=d(x_m,x_*')$, and~$m=d(x_m,y)$. Finally we write~$c=d(x,x_*)$ and~$c'=d(x,x_*')$. All these notations are depicted in Figure~\ref{fig-proof-apollo3} below.
\begin{figure}[h]
\centering
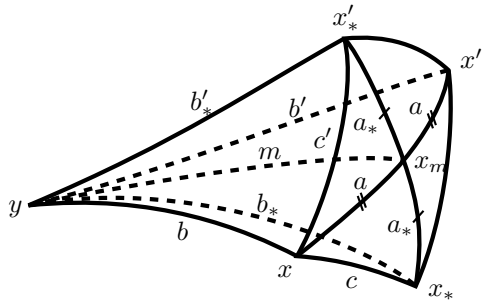
\caption{Configurations as in Lemma~\ref{lem-apollonius-estimate-M-positive} and Lemma~\ref{lem-apollonius-estimate-M-negative} for different triangles.}
\label{fig-proof-apollo3}
\end{figure}

To simplify notation, we consider~$a$,~$b$,~$b'$,~$c$ and~$c'$ as functions of~$x\in B_\varepsilon$, and we will not write explicitly this dependence when no confusion is possible.
We first apply Lemma~\ref{lem-apollonius-estimate-M-positive} in the triangle~$(x,x_*,x_*')$. We get
\[a^2+a_*^2\leqslant\frac{c^2+c'^2}2+C\kappa^{2\varepsilon}\,a_*^2.\]
Integrating with respect to~$K(\cdot,x_*,x_*')$ on~$B_\varepsilon$, we get
\begin{align}
\int_{B_\varepsilon}a^2 K(x,x_*,x_*') \, dx &\leqslant \int_{\mathcal M}\left[\frac{c^2+c'^2}2+C\kappa^{2\varepsilon}\,a_*^2\right]K(x,x_*,x_*') \, dx-\int_{B_\varepsilon}a_*^2 K(x,x_*,x_*') \, dx\nonumber\\
&\leqslant\left(1+\beta+C\kappa^{2\varepsilon}-\int_{B_\varepsilon}K(x,x_*,x_*') \, dx\right)a_*^2,\label{est-a2-K-B}
\end{align}
thanks to Assumption~\ref{Ass:contraction} (which can be rewritten~$\int_{\mathcal M}c^2\,K(x,x_*,x_*') \, dx\leqslant(1+\beta)a_*^2$ with these notations).
Using once again Assumption~\ref{Ass:H3} and the Markov inequality, we get
\begin{align*}
\int_{\mathcal M\setminus B_\varepsilon}K(x,x_*,x_*') \, dx&\leqslant C(\kappa^{-\varepsilon})^p\int_{\mathcal M\setminus B_\varepsilon}d(x,x_*)^p\,K(x,x_*,x_*') \, dx\\
&\leqslant C \kappa^{-p\varepsilon}\, d(x_*,x_*')^p \leqslant C \kappa^{p(1-\varepsilon)},
\end{align*}
and therefore we obtain
\begin{equation}
1-\int_{B_\varepsilon}K(x,x_*,x_*') \, dx\leqslant C \kappa^{p(1-\varepsilon)}.\label{est-K-B}
\end{equation}
Hence, inserting this estimate into~\eqref{est-a2-K-B}, we finally get
\begin{equation}
\int_{B_\varepsilon}a^2 K(x,x_*,x_*') \, dx \leqslant[\,\beta+C(\kappa^{2\varepsilon}+\kappa^{p(1-\varepsilon)})]a_*^2.\label{est-int-a2}
\end{equation}

We now apply Lemma~\ref{lem-apollonius-estimate-M-positive} in the triangle~$(y,x_*,x_*')$, and Lemma~\ref{lem-apollonius-estimate-M-negative} in the triangle~$(y,x,x')$. We get
\begin{gather*}
m^2+a_*^2-\frac{b_*^2+b_*'^2}2\leqslant C \kappa^2 a_*^2\\
m^2+a^2-\frac{b^2+b'^2}2\geqslant-C \kappa^{2\varepsilon} a^2,
\end{gather*}
and therefore we obtain
\[\frac{b^2+b'^2}2-\frac{b_*^2+b_*'^2}2\leqslant(-1+C \kappa^2) a_*^2 + (1+C \kappa^{2\varepsilon}) a^2.\]
Integrating over~$B_\varepsilon$ gives (we recall~\eqref{alpha-epsilon})
\begin{align*}\alpha_\varepsilon(x_*,x_*',y)&\leqslant(-1 + C \kappa^2)a_*^2\int_{B_\varepsilon}K(x,x_*,x_*') \, dx + (1+C \kappa^{2\varepsilon})\int_{B_\varepsilon}a^2K(x,x_*,x_*') \, dx\\
&\leqslant(-1+C\kappa^{p(1-\varepsilon)}+C \kappa^2)a_*^2+(1+C \kappa^{2\varepsilon})[\,\beta+C(\kappa^{2\varepsilon}+\kappa^{p(1-\varepsilon)})]a_*^2,
\end{align*}
Thanks to~\eqref{est-K-B} and~\eqref{est-int-a2}. We finally get
\[\alpha_\varepsilon(x_*,x_*',y)\leqslant(-1+\beta)a_*^2 + C(\kappa^{p(1-\varepsilon)}+\kappa^{2\varepsilon})a_*^2, \]
since~$\varepsilon<1$. Combining this estimate with~\eqref{est-outside-B}, we obtain
\[\alpha(x_*,x_*',y)\leqslant(-1+\beta)a_*^2 + C(\kappa^{(p-2)(1-\varepsilon)}+\kappa^{2\varepsilon})a_*^2, \]
and we see that the best initial choice for~$\varepsilon$ is~$\varepsilon=1-\frac2p$, which gives exactly the estimate~\eqref{estimate-d2-beta}, since~$a_*=\frac12d(x_*,x_*')$.
\end{proof}

The rest of the proof of Theorem~\ref{theo-stability-dirac} is exactly the same as the end of the proof of Theorem~\ref{theo-stability-dirac-midpoint-M}, the different rate of convergence (here~$\frac{1-\beta}4$ instead of~$\frac14$) appearing directly in the local and global estimates provided by Lemmas~\ref{lem-global-beta} and~\ref{lem-local-beta}. 

\section*{Acknowledgements}
The first author is on leave from CNRS, Institut de Mathématiques de Toulouse, UMR 5219, Toulouse. 
He acknowledges support from the ANR under grant MOTIMO (ANR-11-MONU-009-01) and from the National Science Foundation under grant KI-net  \#11-07444 (University of Maryland). 
The first and third author acknowledge support from the CNRS--Royal Society exchange project CODYN. 
The second and third author acknowledge support from the ``ANR blanche'' project Kibord: ANR-13-BS01-0004. The third author acknowledges support from the ``ANR JCJC'' project MODEVOL: ANR-13-JS01-0009.

\begin{bibdiv}
\begin{biblist}

\bib{ballerini2008interaction}{article}{
      author={Ballerini, M.},
      author={Cabibbo, N.},
      author={Candelier, R.},
      author={Cavagna, A.},
      author={Cisbani, E.},
      author={Giardina, I.},
      author={Lecomte, V.},
      author={Orlandi, A.},
      author={Parisi, G.},
      author={Procaccini, A.},
      author={Viale, M.},
      author={Dravkovic, Z.},
       title={Interaction ruling animal collective behavior depends on
  topological rather than metric distance: Evidence from a field study},
        date={2008},
     journal={Proc. Nat. Acad. Sci.},
      volume={105},
      number={4},
       pages={1232\ndash 1237},
}

\bib{bassetti2011central}{article}{
      author={Bassetti, Federico},
      author={Ladelli, Lucia},
      author={Matthes, Daniel},
       title={Central limit theorem for a class of one-dimensional kinetic
  equations},
        date={2011},
        ISSN={0178-8051},
     journal={Probab. Theory Related Fields},
      volume={150},
      number={1-2},
       pages={77\ndash 109},
         url={http://dx.doi.org/10.1007/s00440-010-0269-8},
      review={\MR{2800905 (2012f:60087)}},
}

\bib{berger2003panoramic}{book}{
      author={Berger, Marcel},
       title={A panoramic view of {R}iemannian geometry},
   publisher={Springer-Verlag, Berlin},
        date={2003},
        ISBN={3-540-65317-1},
         url={http://dx.doi.org/10.1007/978-3-642-18245-7},
      review={\MR{2002701 (2004h:53001)}},
}

\bib{bertin2006boltzmann}{article}{
      author={Bertin, E.},
      author={Droz, M.},
      author={Grégoire, G.},
       title={Boltzmann and hydrodynamic description for self-propelled
  particles},
        date={2006},
     journal={Phys. Rev. E},
      volume={74},
       pages={022101},
}

\bib{bertin2009hydrodynamic}{article}{
      author={Bertin, E.},
      author={Droz, M.},
      author={Grégoire, G.},
       title={Hydrodynamic equations for self-propelled particles: microscopic
  derivation and stability analysis},
        date={2009},
     journal={J. Phys. A: Math. Theor.},
      volume={42},
       pages={445001},
}

\bib{bisi2009kinetic}{article}{
      author={Bisi, Marzia},
      author={Spiga, Giampiero},
      author={Toscani, Giuseppe},
       title={Kinetic models of conservative economies with wealth
  redistribution},
        date={2009},
        ISSN={1539-6746},
     journal={Commun. Math. Sci.},
      volume={7},
      number={4},
       pages={901\ndash 916},
         url={http://projecteuclid.org/euclid.cms/1264434137},
      review={\MR{2604625 (2011c:91167)}},
}

\bib{bobylev2000properties}{article}{
      author={Bobylev, A.~V.},
      author={Carrillo, J.~A.},
      author={Gamba, I.~M.},
       title={On some properties of kinetic and hydrodynamic equations for
  inelastic interactions},
        date={2000},
        ISSN={0022-4715},
     journal={J. Statist. Phys.},
      volume={98},
      number={3-4},
       pages={743\ndash 773},
         url={http://dx.doi.org/10.1023/A:1018627625800},
      review={\MR{1749231 (2001c:82063)}},
}

\bib{bolley2012meanfield}{article}{
      author={Bolley, F.},
      author={Cañizo, J.~A.},
      author={Carrillo, J.~A.},
       title={Mean-field limit for the stochastic {Vicsek} model},
        date={2012},
     journal={Appl. Math. Lett.},
      volume={3},
      number={25},
       pages={339\ndash 343},
}

\bib{bolley2013uniform}{article}{
      author={Bolley, François},
      author={Gentil, Ivan},
      author={Guillin, Arnaud},
       title={Uniform convergence to equilibrium for granular media},
        date={2013},
        ISSN={0003-9527},
     journal={Arch. Ration. Mech. Anal.},
      volume={208},
      number={2},
       pages={429\ndash 445},
         url={http://dx.doi.org/10.1007/s00205-012-0599-z},
      review={\MR{3035983}},
}

\bib{boudin2009kinetic}{article}{
      author={Boudin, Laurent},
      author={Salvarani, Francesco},
       title={A kinetic approach to the study of opinion formation},
        date={2009},
        ISSN={0764-583X},
     journal={M2AN Math. Model. Numer. Anal.},
      volume={43},
      number={3},
       pages={507\ndash 522},
         url={http://dx.doi.org/10.1051/m2an/2009004},
      review={\MR{2536247 (2010m:91172)}},
}

\bib{bulmer1980mathematical}{book}{
      author={Bulmer, G.},
       title={The mathematical theory of quantitative genetics},
      series={Oxford science publications},
   publisher={Clarendon Press},
        date={1980},
        ISBN={9780198575306},
         url={http://books.google.fr/books?id=i7pqAAAAMAAJ},
}

\bib{carlen2014model}{unpublished}{
      author={Carlen, E.},
      author={Carvalho, M.},
      author={Degond, P.},
      author={Wennberg, B.},
       title={A model for rod alignment and schooling fish},
        note={submitted},
}

\bib{carlen2013hierarchy}{article}{
      author={Carlen, E.},
      author={Chatelin, R.},
      author={Degond, P.},
      author={Wennberg, B.},
       title={Kinetic hierarchy and propagation of chaos in biological swarm
  models},
        date={2013},
        ISSN={0167-2789},
     journal={Phys. D},
      volume={260},
       pages={90\ndash 111},
         url={http://dx.doi.org/10.1016/j.physd.2012.05.013},
      review={\MR{3143996}},
}

\bib{carlen2013kinetic}{article}{
      author={Carlen, Eric},
      author={Degond, Pierre},
      author={Wennberg, Bernt},
       title={Kinetic limits for pair-interaction driven master equations and
  biological swarm models},
        date={2013},
        ISSN={0218-2025},
     journal={Math. Models Methods Appl. Sci.},
      volume={23},
      number={7},
       pages={1339\ndash 1376},
         url={http://dx.doi.org/10.1142/S0218202513500115},
      review={\MR{3042918}},
}

\bib{carrillo2009double}{article}{
      author={Carrillo, J.~A.},
      author={D'Orsogna, M.~R.},
      author={Panferov, V.},
       title={Double milling in self-propelled swarms from kinetic theory},
        date={2009},
     journal={Kin. Rel. Mod.},
      volume={2},
      number={2},
       pages={363\ndash 378},
}

\bib{carrillo2009overpopulated}{article}{
      author={Carrillo, Jos{\'e}~A.},
      author={Cordier, St{\'e}phane},
      author={Toscani, Giuseppe},
       title={Over-populated tails for conservative-in-the-mean inelastic
  {M}axwell models},
        date={2009},
        ISSN={1078-0947},
     journal={Discrete Contin. Dyn. Syst.},
      volume={24},
      number={1},
       pages={59\ndash 81},
         url={http://dx.doi.org/10.3934/dcds.2009.24.59},
      review={\MR{2476680 (2009m:82025)}},
}

\bib{carrillo2006contractions}{article}{
      author={Carrillo, José~A.},
      author={McCann, Robert~J.},
      author={Villani, Cédric},
       title={Contractions in the 2-{W}asserstein length space and
  thermalization of granular media},
        date={2006},
        ISSN={0003-9527},
     journal={Arch. Ration. Mech. Anal.},
      volume={179},
      number={2},
       pages={217\ndash 263},
         url={http://dx.doi.org/10.1007/s00205-005-0386-1},
      review={\MR{2209130 (2006j:76121)}},
}

\bib{carrillo2007contractive}{article}{
      author={Carrillo, José~Antonio},
      author={Toscani, Giuseppe},
       title={Contractive probability metrics and asymptotic behavior of
  dissipative kinetic equations},
    subtitle={Notes of the 2006 porto ercole summer school},
        date={2007},
     journal={Riv. Mat. Univ. Parma},
      volume={6},
       pages={75\ndash 198},
}

\bib{che2011kinetic}{article}{
      author={Che, Jiahang},
       title={A kinetic model on portfolio in finance},
        date={2011},
        ISSN={1539-6746},
     journal={Commun. Math. Sci.},
      volume={9},
      number={4},
       pages={1073\ndash 1096},
         url={http://dx.doi.org/10.4310/CMS.2011.v9.n4.a7},
      review={\MR{2901817}},
}

\bib{cheeger1975comparison}{book}{
      author={Cheeger, J.},
      author={Ebin, D.~G.},
       title={Comparison theorems in {R}iemannian geometry},
      series={North-Holland Mathematical Library},
   publisher={North-Holland Publishing Company},
     address={Amsterdam},
        date={1975},
      volume={9},
      review={\MR{0458335 (56 \#16538)}},
}

\bib{degond2014phase}{unpublished}{
      author={Degond, P.},
      author={Frouvelle, A.},
      author={Liu, J.-G.},
       title={Phase transitions, hysteresis, and hyperbolicity for
  self-organized alignment dynamics},
        note={preprint arXiv:1304.2929, submitted},
}

\bib{degond2013macroscopic}{article}{
      author={Degond, P.},
      author={Frouvelle, A.},
      author={Liu, J.-G.},
       title={Macroscopic limits and phase transition in a system of
  self-propelled particles},
        date={2013},
     journal={J. Nonlin. Sci.},
      volume={23},
      number={3},
       pages={427\ndash 456},
}

\bib{degond2008continuum}{article}{
      author={Degond, P.},
      author={Motsch, S.},
       title={Continuum limit of self-driven particles with orientation
  interaction},
        date={2008},
     journal={Math. Mod. Meth. Appl. Sci.},
      volume={18},
       pages={1193\ndash 1215},
}

\bib{degond2010macroscopic}{article}{
      author={Degond, P.},
      author={Motsch, S.},
       title={A macroscopic model for a system of swarming agents using
  curvature control},
        date={2011},
     journal={J. Stat. Phys.},
      volume={143},
      number={4},
       pages={685\ndash 714},
}

\bib{frouvelle2012dynamics}{article}{
      author={Frouvelle, A.},
      author={Liu, J.-G.},
       title={Dynamics in a kinetic model of oriented particles with phase
  transition},
        date={2012},
     journal={SIAM J. Math. Anal.},
      volume={44},
      number={2},
       pages={791\ndash 826},
}

\bib{gabetta2012complete}{article}{
      author={Gabetta, Ester},
      author={Regazzini, Eugenio},
       title={Complete characterization of convergence to equilibrium for an
  inelastic {K}ac model},
        date={2012},
        ISSN={0022-4715},
     journal={J. Stat. Phys.},
      volume={147},
      number={5},
       pages={1007\ndash 1019},
         url={http://dx.doi.org/10.1007/s10955-012-0505-y},
      review={\MR{2946634}},
}

\bib{galam1982sociophysics}{article}{
      author={Galam, S},
      author={Gefen, Y},
      author={Shapir, Y},
       title={Sociophysics -- {A} new approach of sociological collective
  behavior {I}. {M}ean-behaviour description of a strike},
        date={1982},
        ISSN={{0022-250X}},
     journal={J. Math. Sociol.},
      volume={9},
      number={1},
       pages={1\ndash 13},
}

\bib{giacomin2012global}{article}{
      author={Giacomin, Giambattista},
      author={Pakdaman, Khashayar},
      author={Pellegrin, Xavier},
       title={Global attractor and asymptotic dynamics in the {K}uramoto model
  for coupled noisy phase oscillators},
        date={2012},
        ISSN={0951-7715},
     journal={Nonlinearity},
      volume={25},
      number={5},
       pages={1247\ndash 1273},
         url={http://dx.doi.org/10.1088/0951-7715/25/5/1247},
      review={\MR{2914138}},
}

\bib{hinow2009analysis}{article}{
      author={Hinow, Peter},
      author={Le~Foll, Frank},
      author={Magal, Pierre},
      author={Webb, Glenn~F.},
       title={Analysis of a model for transfer phenomena in biological
  populations},
        date={2009},
        ISSN={0036-1399},
     journal={SIAM J. Appl. Math.},
      volume={70},
      number={1},
       pages={40\ndash 62},
         url={http://dx.doi.org/10.1137/080732420},
      review={\MR{2505079 (2010d:92112)}},
}

\bib{mischler2009stability}{article}{
      author={Mischler, S.},
      author={Mouhot, C.},
       title={Stability, convergence to self-similarity and elastic limit for
  the {B}oltzmann equation for inelastic hard spheres},
        date={2009},
        ISSN={0010-3616},
     journal={Comm. Math. Phys.},
      volume={288},
      number={2},
       pages={431\ndash 502},
         url={http://dx.doi.org/10.1007/s00220-009-0773-9},
      review={\MR{2500990 (2010i:82160)}},
}

\bib{pareschi2006selfsimilarity}{article}{
      author={Pareschi, Lorenzo},
      author={Toscani, Giuseppe},
       title={Self-similarity and power-like tails in nonconservative kinetic
  models},
        date={2006},
        ISSN={0022-4715},
     journal={J. Stat. Phys.},
      volume={124},
      number={2-4},
       pages={747\ndash 779},
         url={http://dx.doi.org/10.1007/s10955-006-9025-y},
      review={\MR{2264624 (2007h:82077)}},
}

\bib{pasquier2012different}{article}{
      author={Pasquier, Jennifer},
      author={Galas, Ludovic},
      author={Boulange-Lecomte, Celine},
      author={Rioult, Damien},
      author={Bultelle, Florence},
      author={Magal, Pierre},
      author={Webb, Glenn},
      author={Le~Foll, Frank},
       title={Different modalities of intercellular membrane exchanges mediate
  cell-to-cell {P}-glycoprotein transfers in {MCF}-7 breast cancer cells},
        date={2012},
        ISSN={{0021-9258}},
     journal={J. Biol. Chem.},
      volume={287},
      number={10},
       pages={7374\ndash 7387},
}

\bib{pulvirenti2004asymptotic}{article}{
      author={Pulvirenti, Ada},
      author={Toscani, Giuseppe},
       title={Asymptotic properties of the inelastic {K}ac model},
        date={2004},
        ISSN={0022-4715},
     journal={J. Statist. Phys.},
      volume={114},
      number={5-6},
       pages={1453\ndash 1480},
         url={http://dx.doi.org/10.1023/B:JOSS.0000013964.98706.00},
      review={\MR{2039485 (2004k:82090)}},
}

\bib{turelli1994genetic}{article}{
      author={Turelli, M.},
      author={Barton, N.~H.},
       title={Genetic and statistical analyses of strong selection on polygenic
  traits: what, me normal?},
        date={1994},
        ISSN={{0016-6731}},
     journal={Genetics},
      volume={138},
      number={3},
       pages={913\ndash 941},
}

\bib{vicsek1995novel}{article}{
      author={Vicsek, T.},
      author={Czirók, A.},
      author={Ben-Jacob, E.},
      author={Cohen, I.},
      author={Shochet, O.},
       title={Novel type of phase transition in a system of self-driven
  particles},
        date={1995},
     journal={Phys. Rev. Lett.},
      volume={75},
      number={6},
       pages={1226\ndash 1229},
}

\bib{villani2009optimal}{book}{
      author={Villani, C{\'e}dric},
       title={Optimal transport},
    subtitle={Old and new},
      series={Grundlehren der Mathematischen Wissenschaften},
   publisher={Springer-Verlag, Berlin},
        date={2009},
      volume={338},
        ISBN={978-3-540-71049-3},
         url={http://dx.doi.org/10.1007/978-3-540-71050-9},
      review={\MR{2459454 (2010f:49001)}},
}

\end{biblist}
\end{bibdiv}

\end{document}